\documentclass[reqno,twoside]{amsart}
\usepackage{amsfonts}
\usepackage{amsmath,amssymb}

\usepackage[bookmarksopen=true]{hyperref}
\usepackage{graphicx}
\usepackage{color}
\usepackage{subcaption}
\usepackage[outercaption]{sidecap}  
\usepackage{float}

\newcommand{\bea}{\begin{eqnarray}}
\newcommand{\eea}{\end{eqnarray}}

\usepackage{color}

\usepackage{ifpdf}
\ifpdf
  \usepackage{epstopdf}
\fi

\usepackage{a4wide,amsmath,amssymb,latexsym,amsthm}
\usepackage{eucal}
\usepackage{graphicx}

\newtheorem{theorem}{Theorem}[section]

\newtheorem{remark}[theorem]{Remark}
\newtheorem{lemma}[theorem]{Lemma}
 
\newtheorem{definition}[theorem]{Definition}

\numberwithin{equation}{section}

\graphicspath{{./Imgs/}}

\newcommand{\R}{\mathbb R}
 
\newcommand{\N}{\mathbb N}

\newcommand{\V}{\mathbb V}

\newcommand{\be}{\begin{equation}}
\newcommand{\ee}{\end{equation}}
\newcommand{\ba}{\begin{eqnarray}}
\newcommand{\ea}{\end{eqnarray}}
\newcommand{\beq}{\begin{equation}}
\newcommand{\eeq}{\end{equation}}

\numberwithin{equation}{section}

\def\Omc{\R\setminus (-1,1)}

\def\RR{{\mathbb{R}}}
\def\NN{{\mathbb{N}}}
\def\N{{\mathbb{N}}}

\def\V{{\mathbb{V}}}
\def\Om{-1,1}

\def\bbOm{\overline{\Omega}}

\usepackage{color}
\usepackage{stmaryrd}

\newcommand{\fl}[2]{(-\partial_x^{\,2})^#1#2}

\usepackage{color}
\definecolor{red}{rgb}{0,0,0}

\usepackage[textsize=small]{todonotes}


\begin{document}

\title[Nonlocal heat equation]{Controllability properties from the exterior under positivity constraints for a 1-D fractional heat equation}

\author{Harbir Antil}
\address{H. Antil, Department of Mathematical Sciences, George Mason University, Fairfax, VA 22030, USA.}
\email{hantil@gmu.edu}

\author{Umberto Biccari}
\address{U. Biccari,Chair of Computational Mathematics, Fundaci\'on Deusto Av. de las Universidades 24, 48007 Bilbao, Basque Country, Spain.
Facultad de Ingenier\'ia, Universidad de Deusto, Avenida de las Universidades 24, 48007 Bilbao, Basque Country, Spain.}
\email{umberto.biccari@deusto.es, u.biccari@gmail.com}

\author{Rodrigo Ponce}
\address{R. Ponce, Universidad de Talca, Instituto de Matem\'atica y F\'isica, Casilla 747, Talca, Chile.}
\email{rponde@inst-mat.utalca.cl, rodrigo.poncec@gmail.com}

\author{Mahamadi Warma}
\address{M. Warma, University of Puerto Rico, Rio Piedras Campus, Department of Mathematics, Faculty of Natural Sciences,  17 University AVE. STE 1701  San Juan PR 00925-2537 (USA)}
\email{mahamadi.warma1@upr.edu, mjwarma@gmail.com }

\author{Sebasti\'an Zamorano}
\address{S. Zamorano, Universidad de Santiago de Chile, Departamento de Matem\'atica y Ciencia de la Computaci\'on, Facultad de Ciencia, Casilla 307-Correo 2, Santiago, Chile.}
\email{sebastian.zamorano@usach.cl}

\thanks{This project has received funding through UB from the European Research Council (ERC) under the European Union's Horizon 2020 research and innovation programme (grant agreement NO. 694126-DyCon). The work of HA is partially supported by the NSF grants DMS-1818772,
DMS-1913004, and the Air Force Office of Scientific Research under Award NO: FA9550-19-1-0036.
The work UB and MW is partially supported by the Air Force Office of Scientific Research under Award NO: FA9550-18-1-0242. The work of UB is partially supported by the Grant MTM2017-92996-C2-1-R COSNET of MINECO (Spain) and by the ELKARTEK project KK-2018/00083 ROAD2DC of the Basque Government. The work of SZ is supported by the Fondecyt Postdoctoral Grant NO: 3180322
}

\begin{abstract}
We study the controllability to trajectories, under positivity constraints on the control or the state, of a one-dimensional heat equation involving the fractional Laplace operator $ (-\partial_x^2)^s$ (with $0<s<1$) on the interval $(-1,1)$. Our control function is localized in an open set $\mathcal O$ in the exterior of $(-1,1)$, that is, $\mathcal O\subset (\Omc)$. We show that there exists a minimal (strictly positive) time $T_{\rm min}$ such that the fractional heat dynamics can be controlled from any initial datum in  $L^2(-1,1)$ to a positive trajectory through the action of an exterior positive control, if and only if $\frac 12<s<1$. In addition, we prove that at this minimal controllability time, the constrained controllability is achieved by means of a control that belongs to a certain space of Radon measures. Finally, we provide several numerical illustrations that confirm our theoretical results.
\end{abstract}

\maketitle

\section{Introduction}

In this paper we are concerned with the constrained controllability from the exterior of the one-dimensional heat equation associated with the fractional Laplacian on $(-1,1)$. More precisely, we consider the system
\begin{align}\label{eq-main}
	\begin{cases}
		\partial_t u + (-\partial_x^2)^{s} u = 0 & \mbox{ in }\; (-1,1)\times(0,T),
		\\
		u=g\chi_{\mathcal{O}\times(0,T)} &\mbox{ in }\; (\Omc)\times (0,T), 
		\\
		u(\cdot,0) = u_0 &\mbox{ in }\; (-1,1),
	\end{cases}
\end{align}
where $u = u(x,t)$ is the state to be controlled, $0<s<1$ is a real number, $\fl{s}$ denotes the fractional Laplace operator defined for a sufficiently smooth function $v$ by the following singular integral (see Section \ref{sec-2} for more details):  
\begin{align*}
	\fl{s}{v}(x):= C_s\,\mbox{P.V.}\int_{\RR}\frac{v(x)-v(y)}{|x-y|^{1+2s}}\,dy,\;\;x\in \RR,
\end{align*}
and $g = g(x,t)$ is the exterior control function which is localized in a nonempty open subset $\mathcal O$ of $\R\setminus(-1,1)$. 

Our principal goal is to analyze whether the parabolic equation \eqref{eq-main} can be driven from any given initial datum $u_0\in L^2(-1,1)$ to a desired final target by means of the control action, but preserving some non-negativity constraints on the control and/or the state variables. 

The controllability property of fractional heat equations is only recent.
For instance, in \cite{BiHe} it has been shown that, in the absence of constraints, the fractional heat equation is null-controllable from the interior with an $L^2$-control localized in any open set $\omega\subset (-1,1)$, and in any time $T>0$, provided that $\frac 12<s<1$. This has been extended to the constrained controllability case in \cite{biccari2019controllability}, where the authors have shown that the equation is null controllable (hence, controllable to trajectories) with positive $L^\infty$-controls,  for any $\frac 12<s<1$ and any open set $\omega\subset (-1,1)$, provided that the time horizon $T$ for the null-controllability is sufficiently large. The results obtained in  \cite{BiHe,biccari2019controllability} are also valid for the so-called fractional $s$-power of the realization in $L^2(-1,1)$ of the Laplace operator $-\partial_x^2$ with the zero Dirichlet boundary condition. The latter case was first investigated in \cite{Miller}. 

The exterior unconstrained controllability properties of \eqref{eq-main} have been analyzed in  \cite{WaZa1} where the authors obtained analogous results to the ones in the aforementioned papers (that is, null-controllability in any time $T>0$ if and only if $\frac 12<s<1$), but this time by means of an $L^2$-control function acting from the exterior of the domain where the PDE is satisfied. We mention that, as it has been shown in \cite{War-ACE}, a boundary control (that is, the case where the control $g$ is localized in a subset of the boundary) does not make sense for the fractional Laplacian. This is due to the non-locality of the operator and the fact that the fractional heat equation with boundary conditions (Dirichlet, Neumann or Robin) is ill-posed. For this reason, for problems involving the fractional Laplacian the correct notion of a boundary controllability is actually the exterior one, requiring that the control function must be localized outside the domain where the PDE is satisfied, as in the system \eqref{eq-main}. 

For completeness, we also mention that the controllability properties of the fractional heat equation in open subsets of $\RR^N$ ($N\ge 2$) are still not fully understood by the mathematical community. The classical tools (see e.g. \cite{Zua1} and the references therein) like the Carleman estimates usually used to study the controllability for heat equations are still not available for the fractional Laplacian in bounded domains (except in the whole space $\RR^N$). For this reason, in the multi-dimensional case, the best possible controllability result currently available for the fractional heat equation is the approximate controllability recently obtained in \cite{War-ACE} for interior controls and in \cite{WarC} for exterior controls. However, there are multidimensional results on the interior \cite{antil2017optimal} and the exterior optimal control problems \cite{antil2019external,antil2019external1}.

As we said above, the main concern of the present paper is to investigate if it is possible to control from the exterior of $(-1,1)$, the fractional heat dynamics \eqref{eq-main} from any initial datum $u_0\in L^2(-1,1)$ to any positive trajectory $\widehat u$, under positivity constraints on the control and/or the state. This delicate question has been formulated in  \cite{biccari2019controllability} as an open problem.  A complete answer of this question is provided in the present paper. In more detail, the key novelties and the specific results we obtained are as follows: 
\begin{itemize}
	\item[(i)] First, we show in Theorem \ref{main1} that if $\frac 12<s<1$, then the system \eqref{eq-main}
	is controllable from any given initial datum in $L^2(-1,1)$ to zero (and, by translation, to trajectories) in any time $T>0$ by means of $L^\infty$-controls supported in $\mathcal O\subset(\Omc)$. This extends considerably the analysis of \cite{WaZa1}, where only the classical case of $L^2$-controls was considered. The proof will use the canonical approach of reducing the question of controllability with an $L^\infty$-control to a dual observability problem in $L^1$, and the use of Fourier series expansions to obtain a new result on the $L^1$-observation of linear combinations of real exponentials. Notice that, contrary to the case of interior controls, for the exterior control, the $L^1$-observability inequality involves the non-local normal derivative (see \eqref{NLND}) of solutions to the adjoint equation. This normal derivative being a non-local operator makes the problem investigated here more challenging.
	
	\item[(ii)] Secondly, as a consequence of our first result,  in Theorems \ref{main3} and \ref{main4}, we establish the existence of a minimal (strictly positive) time $T_{\rm min}$ such that the fractional heat dynamics \eqref{eq-main} can be controlled to positive trajectories through the action of a positive $L^\infty$-control. Moreover, if the initial datum is assumed to be positive as well, then the maximum principle guarantees the positivity of the states too. 
	
	\item[(iii)] Thirdly, we prove in Theorem \ref{radon measure} that, in the minimal controllability time $T_{\rm min}$, the controllability to positive trajectories holds through the action of a positive control in a space of Radon measures. 
	
	\item[(iv)] Finally, we mention that we have not been successful to have an analytic lower bound of the minimal controllability time $T_{\rm min}$. We accomplish this with the help of some numerical simulations in Section \ref{sec-num}. Notice also that the mentioned numerical simulations shall confirm all our theoretical results. We emphasize that we impose the exterior condition using the approach introduced in \cite{antil2019external,antil2019external1}.
\end{itemize}

In many realistic applications, the control is placed outside the domain where a PDE is fulfilled. Some examples of problems where this \emph{may be of relevance}, noticing that currently local models are used to capture these applications, are:
\begin{itemize}
	\item[(a)] Magnetic drug delivery: the drug with ferromagnetic particles is injected in the body and an external magnetic field is used to steer it to a desired location.
	\item[(b)] Acoustic testing: the aerospace structures are subjected to the sound from the loudspeakers.
\end{itemize}

We refer to \cite{antil2019external,antil2019external1} and their references for a further discussion and the derivation of the exterior optimal control. Let us also mention that the present work is the only available one on constrained controllability properties from the exterior for fractional evolution equations.

Fractional order operators (in particular the fractional Laplace operator) have recently emerged as a modelling alternative in various branches of science.  They usually describe anomalous diffusion. A number of stochastic models for explaining anomalous diffusion have been introduced in the literature. Among them we quote the fractional Brownian motion, the continuous time random walk, the L\'evy flights, the Schneider gray Brownian motion, and more generally, random walk models based on evolution equations of single and distributed fractional order in  space (see e.g. \cite{DS,GR,Man,Sch}). In general, a fractional diffusion operator corresponds to a diverging jump length variance in the random walk. See also \cite{HAntil_CNRautenberg_2019a,CWeiss_BvBWaanders_HAntil_2018a} for the relevance of fractional operators in geophysics and imaging science.

In many PDEs models some constraints need to be imposed when considering concrete applications. This is for instance the case of diffusion processes (heat conduction, population dynamics, etc.) where realistic models have to take into account that the state represents some physical quantity which must necessarily remain positive (see e.g. \cite{chan1990optimal}). This topic is also related to some other relevant applications, like the optimal management of compressors in gas transportation networks requiring the preservation of severe safety constraints (see e.g. \cite{colombo2009optimal,martin2006mixed,steinbach2007pde}). Finally, this issue is also important in other PDEs problems based on scalar conservation laws, including (but not limited to) the Lighthill-Whitham and Richards traffic flow models (\cite{colombo2004minimising,lighthill1955kinematic,richards1956shock}) or the isentropic compressible Euler equation (\cite{glass2007controllability}).

The controllability theory for PDEs has been developed principally without taking into account eventual constraints associated to the phenomenon described by the model under analysis. Actually, to the authors' knowledge, the literature on constrained controllability is currently very limited and the majority of the available results do not guarantee that controlled trajectories fulfill the physical restrictions of the processes under consideration. 

In the context of the local heat equation, the problem of constrained controllability has been addressed in \cite{loheac2017minimal,pighin2018controllability} for the linear and semi-linear cases. In particular, in the mentioned references, the authors proved that, provided the control time is long enough, the linear and semi-linear local heat equations are controllable to any positive steady state or trajectory through the action of non-negative boundary controls. Moreover, for positive initial data, as a consequence of the maximum principle, the positivity of the state is preserved as well. On the other hand, these references, also show the failure of the constrained controllability if the time horizon is too short. 

In addition to the results for heat-like equations, constrained controllability properties have been also analyzed for other classes of parabolic models appearing in the context of population dynamics. In particular, in \cite{hegoburu2018controllability,maity2018controllability}, it has been shown that the controllability of Lotka-McKendrick type systems with age structuring can be obtained by preserving the positivity of the state, once again in a long enough time horizon. These results have been recently extended in \cite{maity2018control} to general infinite-dimensional systems with age structure.

The study of the controllability properties under positivity constraints is a very reasonable question for scalar-valued parabolic equations, which are canonical examples where the positivity is preserved for the free dynamics. Therefore, the issue of whether the system can be controlled in between two states by means of positive controls, by possibly preserving also the positivity of the controlled solution, arises naturally. 

We mention that the existence of a minimal time for constrained controllability may appear non-intuitive with respect to the unconstrained case, in which linear and semi-linear local parabolic systems are known to be controllable at any positive time. However, this is actually not surprising. Indeed, often times, norm-optimal controls allowing to reach the desired target are characterized by large oscillations in the proximity of the final time, which are enhanced when the time horizon of the control is small. This is due to the fact that those controls are restrictions of solutions of the adjoint system, and eventually leads to control trajectories that go beyond the physical thresholds and fail to fulfill the positivity constraint (see \cite{glowinski2008exact}). On the other hand, when the time interval is long, controls of small amplitude are allowed and we may expect the control property to be achieved through small deformations of the state and, in particular, preserving its positivity. 

For completeness, we remark that, in addition to the results for parabolic equations, similar questions for the linear wave equation have been analyzed in \cite{pighin2019controllability}. There, the authors obtained the controllability to steady states and trajectories through the action of a positive control, acting either in the interior or on the boundary of the considered domain. Nevertheless, in that case control and state positivity are not interlinked. Indeed, because of the lack of a maximum principle, the sign of the control does not determine the sign of the solution whose positivity is no longer guaranteed. 

The rest of the paper is organized as follows. In the first part of Section \ref{sec-2}  we fix some notations and state the main results of the paper. The first one (Theorem \ref{main3}) shows that under a positivity constraint on the control, the system \eqref{eq-main} is controllable to trajectories, and in addition, if the initial datum is non-negative, then the state is also non-negative.  The second main result, which is Theorem \ref{main4}, states that the minimal constrained controllability time is strictly positive. Finally our third main result (Theorem \ref{radon measure}) shows that, at the minimal controllability time, the constrained controllability to trajectories is achieved by controls which belong to a certain space of Radon measures. In the second part of Section \ref{sec-2} we recall some known results on fractional parabolic problems as they are needed throughout the article. In Section \ref{sec-null} we prove that there is a control function in $L^\infty(\mathcal O\times (0,T))$ (without any positivity constraint) such that the system \eqref{eq-main} is null controllable in any time $T>0$. Section \ref{prof-ma-re} is devoted to the proofs of our main results. In Section \ref{sec-num} we provide numerical examples that confirm our theoretical findings. Finally, Section \ref{sec-con-rem} is devoted to some final comments and open problems.

\section{Notations, main results and preliminaries}\label{sec-2}

In this section we give some notations, state our main results and recall some known results as they are needed throughout the paper.
We start by introducing the fractional order Sobolev spaces and by giving a rigorous definition of the fractional Laplace operator. 

\subsection{Fractional order Sobolev spaces and the fractional Laplace operator}
Let $\Omega\subset\RR$ be an arbitrary open set. We denote by $C_c(\overline{\Omega})$ the space of all continuous functions with compact support in $\overline{\Omega}$, and for $0<\gamma\le 1$, we let
\begin{align*}
C_c^{0,\gamma}(\overline{\Omega}):=\Big\{u\in C_c(\overline{\Omega}):\;\sup_{x,y\in \Omega,x\ne y}\frac{|u(x)-u(y)|}{|x-y|^{\gamma}}<\infty\Big\}.
\end{align*}
Given $0<s<1$ we define
\begin{align*}
H^{s}(\Omega):=\left\{u\in L^2(\Omega):\;\int_\Omega\int_\Omega\frac{|u(x)-u(y)|^2}{|x-y|^{1+2s}}\;dxdy<\infty\right\},
\end{align*}
and we endow it with the norm given by
\begin{align*}
\|u\|_{H^{s}(\Omega)}:=\left(\int_\Omega|u(x)|^2\;dx+\int_\Omega\int_\Omega\frac{|u(x)-u(y)|^2}{|x-y|^{1+2s}}\;dxdy\right)^{\frac 12}.
\end{align*}
We set
\begin{align*}
\widetilde H_0^{s}(\Omega):=\Big\{u\in H^{s}(\R):\;u=0\;\mbox{ in }\;\R\setminus \Omega\Big\}.
\end{align*}

It is well-known (see e.g. \cite{NPV}) that we have the following continuous embedding: if $\frac 12<s<1$, then
\begin{align} \label{sob-imb-2}
\widetilde H_0^{s}(\Omega)\hookrightarrow C_c^{0,s-\frac 12}(\bbOm).
\end{align}

We shall denote by $\widetilde H^{-s}(\Omega)$ the dual of $\widetilde H_0^{s}(\Omega)$ with respect to the pivot space $L^2(\Omega)$, that is, $\widetilde H^{-s}(\Omega)=(\widetilde H_0^{s}(\Omega))^\star$. In that case we have the following continuous embeddings: $\widetilde H_0^{s}(\Omega)\hookrightarrow L^2(\Omega)\hookrightarrow \widetilde H^{-s}(\Omega)$.
We shall let $\langle\cdot,\cdot\rangle_{-s,s}$ denote their duality pairing. We notice that in most of our results, the open set $\Omega$ will be the bounded open interval $(-1,1)$ or the control region $\mathcal O$.

For more information on fractional order Sobolev spaces, we refer to \cite{NPV,Gris,War} and their references.

Next, we give a rigorous definition of the fractional Laplace operator. Let 
\begin{align*}
\mathcal L_s^{1}(\R):=\left\{u:\R\to\R\;\mbox{ measurable and }\; \int_{\R}\frac{|u(x)|}{(1+|x|)^{1+2s}}\;dx<\infty\right\}.
\end{align*}
For $u\in \mathcal L_s^{1}(\R)$ and $\varepsilon>0$ we set
\begin{align*}
(-\partial_x^2)_\varepsilon^s u(x):= C_{s}\int_{\{y\in\R:\;|x-y|>\varepsilon\}}\frac{u(x)-u(y)}{|x-y|^{1+2s}}\;dy,\;\;x\in\R,
\end{align*}
where $C_{s}$ is a normalization constant given by
\begin{align}\label{CNs}
C_{s}:=\frac{s2^{2s}\Gamma\left(\frac{2s+1}{2}\right)}{\pi^{\frac
12}\Gamma(1-s)}.
\end{align}
The {\em fractional Laplacian}  $(-\partial_x^2)^s$ is defined by the following singular integral:
\begin{align}\label{fl_def}
(-\partial_x^2)^su(x):=C_{s}\,\mbox{P.V.}\int_{\R}\frac{u(x)-u(y)}{|x-y|^{1+2s}}\;dy=\lim_{\varepsilon\downarrow 0}(-\partial_x^2)_\varepsilon^s u(x),\;\;x\in\R,
\end{align}
provided that the limit exists for a.e. $x\in\RR$. We notice that $\mathcal L_s^{1}(\R)$ is the right space for which $ v:=(-\partial_x^2)_\varepsilon^s u$ exists for every $\varepsilon>0$, $v$ being also continuous at the continuity points of  $u$.  For more details on the fractional Laplace operator we refer to \cite{Caf3,NPV,GW-CPDE,War} and their references.

Next, we consider the realization of $(-\partial_x^2)^s$ in $L^2(-1,1)$ with the exterior zero Dirichlet condition. More precisely, we consider the closed and bilinear form $\mathcal F: \widetilde H_0^{s}(\Om)\times  \widetilde H_0^{s}(\Om)\to \RR$ given by
\begin{align}\label{closed}
\mathcal F(u,v):=\frac{C_{s}}{2}\int_{\R}\int_{\R}\frac{(u(x)-u(y))(v(x)-v(y))}{|x-y|^{1+2s}}\;dxdy,\;\;u,v\in \widetilde H_0^{s}(\Om).
\end{align}

Let $(-\partial_x^2)_D^s$ be the self-adjoint operator on $L^2(-1,1)$ associated with $\mathcal F$ in the sense that
\begin{equation*}
\begin{cases}
D((-\partial_x^2)_D^s):=\Big\{u\in \widetilde H_0^{s}(\Om):\;\exists\;f\in L^2(-1,1)\mbox{ and }\;\mathcal F(u,v)=(f,v)_{L^2(-1,1)}\\
\qquad\qquad\qquad\quad\;\forall\;v\in \widetilde H_0^{s}(\Om)\Big\},\\
(-\partial_x^2)_D^su:=f.
\end{cases}
\end{equation*}
We have that  (see e.g. \cite{BC-MW,SV2})
\begin{equation}\label{DLO}
\begin{cases}
D((-\partial_x^2)_D^s):=\left\{u\in\widetilde H_0^{s}(\Om):\; (-\partial_x^2)^su\in L^2(-1,1)\right\},\\
(-\partial_x^2)_D^su:=((-\partial_x^2)^su)|_{(-1,1)}.
\end{cases}
\end{equation}

Then, $(-\partial_x^2)_D^s$ is the realization of $(-\partial_x^2)^s$ in $L^2(-1,1)$ with the condition $u=0$ in $\Omc$. By \cite{SV2}, $(-\partial_x^2)_D^s$ has a compact resolvent and its eigenvalues form a non-decreasing sequence of real numbers $0<\lambda_1\le\lambda_2\le\cdots\le\lambda_n\le\cdots$ satisfying $\lim_{n\to\infty}\lambda_n=\infty$.  In addition, the eigenvalues are of finite multiplicity and are simple if $\frac 12\le s<1$.
Let $(\varphi_n)_{n\in\NN}$ be the orthonormal basis of eigenfunctions associated with $(\lambda_n)_{n\in\NN}$. Then, $\varphi_n\in D((-\partial_x^2)_D^s)$ for every $n\in\NN$,  $(\varphi_n)_{n\in\NN}$ is total in $L^2(-1,1)$ and satisfies 
\begin{equation}\label{ei-val-pro}
\begin{cases}
(-\partial_x^2)^s\varphi_n=\lambda_n\varphi_n\;\;&\mbox{ in }\;(-1,1),\\
\varphi_n=0\;&\mbox{ in }\;\Omc.
\end{cases}
\end{equation}
Next, for $u\in H^{s}(\R)$ we introduce the {\em nonlocal normal derivative $\mathcal N_s$} given by 
\begin{align}\label{NLND}
\mathcal N_su(x):=C_{s}\int_{-1}^1\frac{u(x)-u(y)}{|x-y|^{1+2s}}\;dy,\;\;\;x\in(\R \setminus[-1,1]),
\end{align}
where $C_{s}$ is the constant given in \eqref{CNs}.
Since equality is to be understood a.e., we have that \eqref{NLND} is the same as for a.e. $x\in\Omc$.

The following unique continuation property, which shall play an important role in the proof of our main results, has been recently obtained in \cite[Theorem 16]{War-ACE}.

\begin{lemma}\label{lem-UCD}
Let $\lambda>0$ be a real number  and $\mathcal O\subset(\Omc)$ an arbitrary  nonempty open set. 
If $\varphi\in D((-\partial_x^2)_D^s)$ satisfies
\begin{equation*}
(-\partial_x^2)_D^s\varphi=\lambda\varphi\;\mbox{ in }\;(-1,1)\;\mbox{ and }\; \mathcal N_s\varphi=0\;\mbox{ in }\;\mathcal O,
\end{equation*}
then $\varphi=0$ in $\R$. 
\end{lemma}

For more details on the Dirichlet problem associated with the fractional Laplace operator we refer the interested reader to \cite{BWZ1,Grub,RS2-2,RS-DP,War-ACE} and their references.

We conclude this section with the following integration by parts formula.

\begin{lemma}
Let $u\in \widetilde H_0^{s}(\Om)$ be such that $(-\partial_x^2)^su\in L^2(-1,1)$ and $\mathcal N_su\in L^2(\Omc)$. Then, the identity
\begin{multline}\label{Int-Part}
\frac{C_{s}}{2}\int_{\RR}\int_{\RR}\frac{(u(x)-u(y))(v(x)-v(y))}{|x-y|^{1+2s}}\;dxdy=\int_{-1}^1v(x)(-\partial_x^2)^su(x)\;dx\\+\int_{\Omc}v(x)\mathcal N_su(x)\;dx,
\end{multline}
holds for every $v\in H^s(\RR)$.
\end{lemma}

We refer to \cite[Lemma 3.3]{DRV} (see also \cite{War-ACE,WaZa}) for the proof and for more details.

\subsection{Main results}\label{main-results}
In this section we state the main results of the paper. We start with our controllability to trajectories result of the system \eqref{eq-main} with $L^\infty$-controls and positivity constraints.

\begin{theorem}\label{main3}
Let $\mathcal O\subset (\Omc)$ be an arbitrary nonempty bounded open set.
Let $\frac 12<s<1$ and consider a positive trajectory $\widehat{u}$ of \eqref{eq-main} with initial datum $0<\widehat{u}_0\in L^2(-1,1)$ and exterior control datum $\widehat{g}\in L^{\infty}(\mathcal{O}\times(0,T))\cap L^2((0,T);\widetilde H_0^s(\mathcal{O}))$ for which there is a positive constant $\alpha$ such that $\widehat{g}\geq\alpha$ a.e. in $\mathcal O\times (0,T)$. Then, there exist $T>0$ and a non-negative control $g\in L^{\infty}(\mathcal{O}\times(0,T))\cap L^2((0,T);\widetilde H_0^s(\mathcal{O}))$ such that the corresponding weak solution $u$ of \eqref{eq-main} satisfies $u(\cdot,T)=\widehat{u}(\cdot,T)$ a.e. in $(-1,1)$. In addition, if $u_0\geq 0$ a.e. in $(-1,1)$, then $u\geq 0$ a.e. in $ (-1,1)\times (0,T)$.
\end{theorem}

\begin{remark}
{\em We notice that the assumption that the control region $\mathcal O$ must be bounded, is necessary to ensure that the control  $\widehat{g}\in L^{\infty}(\mathcal{O}\times(0,T))$ satisfying $\widehat{g}\geq\alpha$ a.e. in $\mathcal O\times (0,T)$ also belongs to $L^2((0,T);\widetilde H_0^s(\mathcal{O}))$.}
\end{remark}

Our second main result, which is the following theorem, shows that the minimal controllability time is strictly positive.

\begin{theorem}\label{main4}
Under the hypothesis of Theorem \ref{main3}, let $T_{\min}$ be the minimal controllability time given by
\begin{multline}\label{minimaltime}
T_{\min}:=\inf\Big\{T>0\; :\; \exists\; 0\leq g\in L^{\infty}(\mathcal{O}\times(0,T))\cap L^2((0,T);\widetilde H_0^s(\mathcal{O}))\\ \text{ such that }u(\cdot,T)=\widehat{u}(\cdot,T)\Big\}.
\end{multline}
Then, $T_{\min}>0$.
\end{theorem}

Next, let $\mathcal{M}(\mathcal{O}\times(0,T))$ be the space of Radon measures on $\mathcal O\times (0,T)$. Then  $\mathcal{M}(\mathcal{O}\times(0,T))$ endowed with the norm
\begin{multline*}
\|\mu\|_{\mathcal{M}(\mathcal{O}\times(0,T))}:=\sup\Bigg\{\int_{\mathcal{O}\times(0,T)} \xi(x,t) d\mu(x,t)\; : \; \xi\in C(\overline{\mathcal{O}}\times[0,T],\RR),\\\; \max_{\overline{\mathcal{O}}\times[0,T]}|\xi|=1  \Bigg\},
\end{multline*}
is a Banach space. Here we assume that the control region $\mathcal O$ is bounded.

Our last main result shows that at the  minimal controllability time $T_{\rm min}$, the null-controllability of the system \eqref{eq-main} is achieved with controls in $\mathcal{M}(\mathcal{O}\times(0,T))$.

\begin{theorem}\label{radon measure}
Let the hypothesis of Theorem \ref{main3} hold and assume in addition that $\mathcal O\subset\overline{\mathcal O}\subset(\R\setminus[-1,1])$. Let $T:=T_{\min}$ be the minimal controllability time given by \eqref{minimaltime}. Then, there exists a non--negative control $g\in \mathcal{M}(\mathcal{O}\times(0,T))$ such that the corresponding solution $u$ of \eqref{eq-main} satisfies $u(\cdot,T)=\widehat{u}(\cdot,T)$ a.e. in $(-1,1)$.
\end{theorem}

\subsection{Well-posedness of the parabolic problems}\label{sec-3}

In this section we collect some well-known results contained in \cite{War-ACE,WaZa1} regarding the well-posedness and the series representation of solutions to the system \eqref{eq-main} and the associated dual system. In addition, we shall recall the maximum principle for fractional heat equations.

Throughout the remainder of the article, without any mention, $(\varphi_n)_{n\in\NN}$ denotes the orthonormal basis of eigenfunctions of the operator $(-\partial_x^2)_D^s$ associated with the eigenvalues $(\lambda_n)_{n\in\NN}$. If $u\in L^2(-1,1)$, then we shall let $u_n:=(u,\varphi_n)_{L^2(-1,1)}$.
Furthermore, for a given measurable set $E\subseteq \RR^N$ ($N\ge 1$), we shall denote by $(\cdot,\cdot)_{L^2(E)}$ the scalar product in $L^2(E)$. 

Next, we introduce our notion of weak solutions.

\begin{definition}
Let $g\in L^{\infty}(\RR\setminus(-1,1))\times(0,T))\cap L^2((0,T);H^s(\RR\setminus(-1,1)))$ and $h\in L^{\infty}(\RR\times(0,T))\cap L^2((0,T);H^s(\RR))$ be such that $h|_{\RR\setminus(-1,1)}=g$. We shall say that a function 
$$u\in C([0,T];L^2(-1,1))\cap L^2((0,T);H^s(\RR))\cap H^1((0,T);\widetilde H^{-s}(\Om))$$
 is a weak solution of the system \eqref{eq-main}, if 
 $$(u-h)\in L^2((0,T);\widetilde H_0^s(\Om))\cap H^1((0,T);H^{-s}(-1,1)), \; u(\cdot,0)=u_0\mbox{ a.e. in } (-1,1),$$
 and the identity
\begin{align*}
\langle u_t(\cdot,t), v\rangle_{-s,s} +\frac{C_s}{2}\int_{\RR}\int_{\RR}\frac{(u(x,t)-u(y,t))(v(x)-v(y))}{|x-y|^{1+2s}}dxdy=0
\end{align*}
holds for every $v\in \widetilde H_0^s(-1,1)$ and almost every $t\in(0,T)$.
\end{definition}

We have the following existence result and the explicit representation of solutions in terms of series.
The proof can be found in \cite{War-ACE,WaZa1}.

\begin{theorem}\label{expl-2}
Let $\mathcal O\subset\Omc$ be an arbitrary non-empty open set.
Then, for every $u_0\in L^2(-1,1)$ and $g\in L^{\infty}(\mathcal{O}\times(0,T))\cap L^2((0,T);\widetilde H_0^s(\mathcal{O}))$, the system \eqref{eq-main} has a unique weak solution $u$ given by
\begin{align}\label{4}
u(x,t)=\sum_{n=1}^{\infty}u_{0,n}e^{-\lambda_n t}\varphi_n(x)+\sum_{n=1}^{\infty}\left(\int_0^t (g(\cdot,\tau), \mathcal{N}_s\varphi_n)_{L^2(\mathcal{O})}  e^{-\lambda_n(t-\tau)}d\tau\right)\varphi_n(x).
\end{align}
\end{theorem}

Using the classical integration by parts formula, we have that the following backward system
\begin{equation}\label{Dual}
\begin{cases}
-\partial_t \psi +(-\partial_x^2)^s\psi=0\;\;&\mbox{ in }\; (-1,1)\times (0,T),\\
\psi=0&\mbox{ in }\;(\Omc)\times (0,T),\\
\psi(\cdot,T)=\psi_T&\mbox{ in }\;\Omega,
\end{cases}
\end{equation} 
can be viewed as the dual system associated with \eqref{eq-main}.

\begin{definition}
Let $\psi_T\in L^2(-1,1)$.
By a weak solution to \eqref{Dual}, we mean a function 
$$\psi\in C([0,T];L^2(-1,1))\cap L^2((0,T);\widetilde H_0^s(-1,1))\cap H^1((0,T);\widetilde H^{-s}(-1,1)),$$ 
such that $\psi(\cdot,T)=\psi_T$ a.e. in $(-1,1)$, and the identify
\begin{align*}
-\langle \psi_t(\cdot,t), v\rangle_{-s,s} +\frac{C_s}{2}\int_{\RR}\int_{\RR}\frac{(\psi(x,t)-\psi(y,t))(v(x)-v(y))}{|x-y|^{1+2s}}dxdy=0
\end{align*}
holds for every $v\in \widetilde H_0^s(-1,1)$ and almost every $t\in(0,T)$.
\end{definition}

We have the following existence result (see e.g. \cite{War-ACE,WaZa1}).

\begin{theorem}\label{theo-48}
Let $\psi_T\in  L^2(0,1)$. Then, the dual system \eqref{Dual} has a unique weak solution $\psi$ which is given by
\begin{align}\label{eq-25}
\psi(x,t)=\sum_{n=1}^{\infty}\psi_{0,n}e^{-\lambda_n(T-t)}\varphi_{n}(x).
\end{align}
In addition the following assertions hold.
\begin{enumerate}
\item There is a constant $C>0$ such that for all $t\in [0,T]$,
\begin{equation*}
 \|\psi(\cdot,t)\|_{L^2(-1,1)}\le C\|\psi_T\|_{L^2(-1,1)}.
\end{equation*}
\item For every $t\in[0,T)$ fixed, $\mathcal{N}_s \psi(\cdot,t)$ exists, belongs to $L^2(\Omc)$ and is given by
\begin{align}\label{norm-der}
\mathcal{N}_s \psi(x,t)=\sum_{n=1}^{\infty}\psi_{0,n}e^{-\lambda_n(T-t)}\mathcal{N}_s\varphi_{n}(x).
\end{align}
\end{enumerate}
In \eqref{eq-25} and \eqref{norm-der} we have set $\psi_{0,n}:=(\psi_T,\varphi_n)_{L^2(-1,1)}$.
\end{theorem}

We conclude this section with the comparison principle taken from \cite[Corollary 2.11]{andreu2010nonlocal}. This will be used in the proof of our main results.

\begin{theorem}\label{theo-28}
Let $u_0$ and $v_0$ be such that $u_0\geq v_0$ a.e. in $(-1,1)$ and let $g,h$ be such that $g\geq h$ a.e. in $(\Omc)\times (0,T)$. Let $u$  be the weak solution of \eqref{eq-main} with initial datum $u_0$ and exterior datum $g$.  Let $v$  be the weak solution of \eqref{eq-main} with initial datum $v_0$ and exterior datum $h$. Then $u\geq v$ a.e. in $(-1,1)\times (0,T)$.
\end{theorem}

\section{Null controllability with $L^{\infty}$-controls without constraints}\label{sec-null}

In this section we analyze the null controllability properties of  \eqref{eq-main} with controls in $L^{\infty}(\mathcal{O}\times(0,T))\cap L^2((0,T); \widetilde H_0^s(\mathcal{O}))$ but without imposing any positivity constraint on the control and/or the state.  These results  shall play a crucial role in the proofs of our main results.

We start by introducing our notion of  null controllability of the system \eqref{eq-main} and an $L^1$-observability inequality for the associated dual system \eqref{Dual}.

\begin{definition}
We say that the system \eqref{eq-main} is null controllable in time $T>0$, if for every $u_0\in L^2(-1,1)$, there exists a control function $g\in L^{\infty}(\mathcal{O}\times(0,T))\cap L^2((0,T);\widetilde H_0^s(\mathcal{O}))$ such that the associated unique weak solution $u$ satisfies
\begin{align}\label{control}
u(x,T)=0\quad \text{ for a.e. }x\in(-1,1).
\end{align}
\end{definition}

\begin{definition}
The system \eqref{Dual} is said to be $L^1$-observable in time $T>0$, if there exists a constant $C=C(T)>0$ such that the inequality 
\begin{align}\label{observ}
\|\psi(\cdot,0)\|_{L^2(-1,1)}^2\leq C\left(\int_0^T\int_{\mathcal{O}}|\mathcal{N}_s\psi(x,t)|dxdt\right)^2
\end{align}
holds for every $\psi_T\in L^2(-1,1)$, where $\psi$ is the unique weak solution of \eqref{Dual} with final datum $\psi_T$, and $\mathcal{N}_s\psi$ is the nonlocal normal derivative of $\psi$ given in \eqref{norm-der}.
\end{definition}

We have the following result.

\begin{theorem}\label{main1}
Let $\mathcal O\subset(\Omc)$ be an arbitrary nonempty open set.
Then the following assertions are equivalent.
\begin{enumerate}
\item For every $u_0\in L^2(-1,1)$ and $T>0$, the system \eqref{eq-main} is null controllable in time $T>0$. Moreover, there is a constant $C_1>0$ ( independent of $u_0$) such that the control $g$ satisfies the following estimate:
\begin{align}\label{bound of control}
\|g\|_{L^{\infty}(\mathcal{O}\times(0,T))}\le \|g\|_{L^\infty(\mathcal O\times (0,T))\cap L^2((0,T);\widetilde H_0^s(\mathcal O))} \leq C_1 \|u_0\|_{L^2(-1,1)}.
\end{align}

\item For every $T>0$ and  $\psi_T\in L^2(-1,1)$ the dual system \eqref{Dual} is $L^1$-observable.
\end{enumerate}
\end{theorem}

\begin{proof}
(1) $\Rightarrow$ (2): Assume that \eqref{eq-main} is null controllable in time $T>0$. Then there exists a control function $g\in L^{\infty}(\mathcal{O}\times(0,T))\cap L^2((0,T);\widetilde H_0^s(\mathcal{O}))$ such that \eqref{control} holds.  Let $\psi$ be the unique weak solution of  \eqref{Dual} with $\psi_T\in L^2(-1,1)$. Multiplying \eqref{eq-main} with $\psi$, integrating over $(-1,1)\times(0,T)$ and using \eqref{Int-Part}, we get that
\begin{align}\label{mw}
\int_{-1}^1u_0(x)\psi(x,0)dx=\int_0^T\int_{\mathcal{O}}g(x,t)\mathcal{N}_s\psi(x,t)dxdt.
\end{align}
Letting $u_0(x):=\psi(x,0)$ in \eqref{mw} and using the H\"older inequality we obtain  that
\begin{multline}\label{mw2}
\int_{-1}^1|\psi(x,0)|^2dx=\int_0^T\int_{\mathcal{O}}g(x,t)\mathcal{N}_s\psi(x,t)dxdt
\\\leq \|g\|_{L^{\infty}(\mathcal{O}\times(0,T))}\int_0^T\int_{\mathcal{O}}|\mathcal{N}_s\psi(x,t)|dxdt.
\end{multline}
Using \eqref{bound of control} and Young's inequality we get from \eqref{mw2} that 
\begin{align}\label{mw3}
\int_{-1}^1|\psi(x,0)|^2dx\leq \frac{C_1}{2\varepsilon}\|u_0\|_{L^2(-1,1)}^2+\frac{\varepsilon}{2}\left(\int_0^T\int_{\mathcal{O}}|\mathcal{N}_s\psi(x,t)|dxdt\right)^2
\end{align}
for every $\varepsilon>0$.  Taking $\varepsilon:=C_1$ in \eqref{mw3} and since $u_0(x)=\psi(x,0)$, we can deduce that \eqref{observ} holds. 

(2) $\Rightarrow$ (1): We have to show that \eqref{observ} implies the null controllability of \eqref{eq-main}. For every $\psi_T\in L^2(-1,1)$ and $u_0\in L^2(-1,1)$ we have that
\begin{align}\label{wj}
\int_{-1}^1u_0(x)\psi(x,0)dx-\int_{-1}^1u(x,T)\psi_T(x)dx=\int_0^T\int_{\mathcal{O}}g(x,t)\mathcal{N}_s\psi(x,t)dxdt.
\end{align}

Let us consider the linear subspace $\Lambda$ of $\mathbb X:=L^1(\mathcal{O}\times(0,T))\cap L^2((0,T);\widetilde H^{-s}(\mathcal{O}))$ given by:
\begin{align*}
\Lambda:=\Big\{\mathcal{N}_s\psi\Big|_{\mathcal{O}\times(0,T)}\; : \; \psi \text{ solves }\eqref{Dual} \text{ with }\psi_T\in L^2(-1,1)\Big\}.
\end{align*}
Let $u_0\in L^2(-1,1)$ and consider the linear functional $F: \Lambda\to \RR$ defined by
\begin{align*}
F(\mathcal{N}_s\psi):=(u_0,\psi(\cdot,0))_{L^2(-1,1)}.
\end{align*}

 It follows from \eqref{observ} that $F$ is well defined and bounded on $\Lambda$. Namely, using \eqref{observ} we get that
\begin{align*}
	|F(\mathcal{N}_s\psi)| &\le \|u_0\|_{L^2(-1,1)}\|\psi(\cdot,0)\|_{L^2(-1,1)}
	\\
	&\leq C\|u_0\|_{L^2(-1,1)}\|\mathcal{N}_s\psi\|_{L^1(\mathcal{O}\times(0,T))}\leq C\|u_0\|_{L^2(-1,1)}\|\mathcal{N}_s\psi\|_{\mathbb X}. 
\end{align*}

By the Hahn-Banach Theorem, $F$ can be extended to a bounded linear functional $\widetilde{F}:\mathbb X\to\RR$ such that 
\begin{align*}
|\widetilde{F}v|\leq C_1\|u_0\|_{L^2(-1,1)}\|v\|_{\mathbb X}, \quad\forall \;v\in \mathbb X.
\end{align*}

By the Riesz representation Theorem, there is a $g\in \mathbb X^\star=L^{\infty}(\mathcal{O}\times(0,T))\cap L^2((0,T);\widetilde H_0^s(\mathcal{O}))$ such that
\begin{align*}
\|g\|_{L^{\infty}(\mathcal{O}\times(0,T))}\le \|g\|_{\mathbb X^\star}\leq C_1\|u_0\|_{L^2(-1,1)}
\end{align*}
and
\begin{align}\label{JW}
\widetilde{F}(\xi)=\int_0^T\int_{\mathcal{O}}g(x,t)\xi(x,t)dxdt,\quad\forall\;\xi\in\mathbb X.
\end{align}
Notice that $\mathcal N_s\psi\in \mathbb X$. Thus, using the definition of $F$ we get from \eqref{JW} that 
\begin{align*}
F(\mathcal{N}_s\psi)=\int_0^T\int_{\mathcal{O}}g_2(x,t)\mathcal{N}_s\psi(x,t)dxdt=(u_0,\psi(\cdot,0))_{L^2(-1,1)},
\end{align*}
for every $\psi_T\in L^2(-1,1)$. We have shown that there is a control $g\in \mathbb X^\star$ such that \eqref{bound of control} is satisfied and 
\begin{align}\label{new1}
\int_{-1}^1u_0(x)\psi(x,0)dx=\int_0^T\int_{\mathcal{O}}g(x,t)\mathcal{N}_s\psi(x,t)dxdt
\end{align}
for every $\psi_T\in L^2(-1,1)$. It follows from \eqref{wj} and \eqref{new1} that
$\displaystyle\int_{-1}^1u(x,T)\psi_T(x)dx=0$
for every $\psi_T\in L^2(-1,1)$.  Thus, $u(x,T)=0$ for a.e. $x\in (-1,1)$. The proof is finished.
\end{proof}

The results in Theorem \ref{main1} show that, in order to obtain the null controllability of the system \eqref{eq-main}, it is enough to prove the $L^1$-observability inequality \eqref{observ}.
To do this, we need first to establish some auxiliaries results.

We start with the following Ingham-type one recently obtained in \cite[Theorem 2.4]{biccari2019controllability}. 

\begin{theorem}\label{muntz1}
Let $(\mu_n)_{n\geq1}\subset [0,\infty)$ be a sequence satisfying the following conditions:
\begin{enumerate}
\item There exists $\gamma>0$ such that $\mu_{n+1}-\mu_{n}\geq\gamma$ for all $n\geq1$.
\item $\displaystyle\sum_{n\geq1}\frac{1}{\mu_n}<\infty$.
\end{enumerate}
Then, for any $T>0$, there is a constant $C(T)>0$ such that, for any sequence $(c_{n})_{n\geq1}$ of numbers it holds the inequality:
\begin{align}\label{muntz}
\sum_{n\geq1}|c_n|e^{-\mu_n T}\leq C(T)\left\|\sum_{n\geq1}c_n e^{-\mu_n t}\right\|_{L^1(0,T)}.
\end{align}
Moreover, $C(T)$ is uniformly bounded away from $T=0$ and blows-up exponentially as $T\downarrow 0^+$.
\end{theorem}

The second auxiliary and technical result we shall need, is adapted from the results contained in \cite{WaZa1}. In fact, by \cite{WaZa1}, $\|\mathcal N_s\varphi_n\|_{L^2(\mathcal O)}$ is uniform bounded from below, where $\mathcal{O}\subset(\Omc)$ is an arbitrary open. In the settings of the present paper, we shall need a similar estimate but for the  $L^1$-norm.

\begin{lemma}\label{uni-bound}
Let $\frac 12<s<1$. Then, for every nonempty open set $\mathcal{O}\subset (\Omc)$, there exists a constant $\eta>0$ such that for every $k\in\N$,
$\mathcal{N}_s\varphi_k$ is uniformly bounded from below by $\eta$ in $L^1(\mathcal{O})$. Namely,
\begin{align}\label{eq39}
\exists\; \eta>0,\; \forall \;k\in\N,\; \|\mathcal{N}_s\varphi_k\|_{L^1(\mathcal{O})}\geq \eta.
\end{align}
\end{lemma}

\begin{proof}
For brevity we present here only the main ideas of the proof.  Let $\frac 12<s<1$.\\

{\bf Step 1}:
Since $\varphi_k=0$  in $\Omc$ for every $k\in\NN$, it follows from the definition of $(-\partial_x^2)^s$ and $\mathcal N_s$ that for almost every $x\in\mathcal O\subseteq(\Omc)$, we have 
\begin{align}\label{eq-eg}
(-\partial_x^2)^s\varphi_k(x)=C_{s}\mbox{P.V.}\int_{\RR}\frac{\varphi_k(x)-\varphi_k(y)}{|x-y|^{1+2s}}\;dy=C_{s}\int_{-1}^1\frac{\varphi_k(x)-\varphi_k(y)}{|x-y|^{1+2s}}\;dy=
\mathcal N_s\varphi_k(x).
\end{align}
We have shown that $(\mathcal N_s \varphi_k)|_{\mathcal O}=((-\partial_x^2)^s\varphi_k)|_{\mathcal O}$ for every $k\in\NN$. 

It follows from \cite[Lemma 1]{kwasnicki2012eigenvalues} that $(\varphi_k)_{k\geq1}$ can be approximated by a suitable sequence $(\varrho_k)_{k\in\NN}\subset D((-\partial_x^2)_D^s)$, and  there is a constant $C>0$ (independent of $k$) such that
\begin{align*}
\left| \fl{s}{\varrho_k}(x)-\mu_k^{2s}\varrho_k(x)\right|\leq \frac{C(1-s)}{\sqrt{2s}}\mu_k^{-1},\;\;\textrm{ for all }\;\; x\in (-1,1),\;k\geq 1,
\end{align*}
where 
\begin{align*} \label{mu_k_def}
\mu_k:=\frac{k\pi}{2}-\frac{(1-s)\pi}{4},\quad k\geq 1.
\end{align*}

Furthermore, by \cite[Proposition 1]{kwasnicki2012eigenvalues}, there is a constant $C>0$ such that for every $k\ge 1$, we have
\begin{equation*}
\|\varrho_k-\varphi_k\|_{L^2(-1,1)}\le \frac{C(1-s)}{k}.
\end{equation*}

{\bf Step 2:}
Now, let $\mathcal{O}\subset \Omc$ be an arbitrary nonempty open  set and assume that for every $\eta>0$ there exists $k\in\N$ such that 
\begin{align}\label{contr}
\|\mathcal{N}_s\varphi_k\|_{L^1(\mathcal{O})}<\eta.
\end{align}
It follows from \eqref{contr} that there is a subsequence $(\varphi_{k_n})_{n\in\NN}$ such that
\begin{align}\label{contr-2}
\|\mathcal{N}_s\varphi_{k_n}\|_{L^1(\mathcal{O})}<\frac 1n,
\end{align}
for $n$ large enough. Since $\frac 12<s<1$, it follows from \eqref{sob-imb-2} that $\widetilde H_0^s(\mathcal{O})\hookrightarrow L^{\infty}(\mathcal{O})$. Thus, $L^1(\mathcal{O})\hookrightarrow (L^{\infty}(\mathcal{O}))^{\star}\hookrightarrow  \widetilde H^{-s}(\mathcal{O})$ and we can deduce from \eqref{contr-2} that there is a constant $C>0$ such that for $n$ large enough, we have
\begin{align}\label{contr-3}
\|\mathcal{N}_s\varphi_{k_n}\|_{\widetilde H^{-s}(\mathcal{O})}\le \frac Cn.
\end{align}

{\bf Step 3}: Using the triangle inequality, we get that there is a constant $C>0$ such that
\begin{align}\label{nnn}
\|\varrho_{k_n}-\varphi_{k_n}\|_{\widetilde H_0^s(\Om)}^2\leq &C\|(-\partial_x^2)^s\varrho_{k_n}-(-\partial_x^2)^s\varphi_{k_n}\|_{L^2(-1,1)}^2\notag\\
\leq &C \Big(\|(-\partial_x^2)^s\varrho_{k_n}-\mu_{k_n}^{2s}\varrho_{k_n}\|_{L^2(-1,1)}^2+\|\varrho_{k_n}(\mu_{k_n}^{2s}-\lambda_{k_n})\|_{L^2(-1,1)}^2\notag \\
&+\|\lambda_{k_n}\varrho_{k_n}-(-\partial_x^2)^s\varphi_{k_n}\|_{L^2(-1,1)}^2\Big).
\end{align}

Using \eqref{nnn} and  Step 1, we have that there is a constant $C_{k_n}(s)>0$ which converges to zero as $n\to\infty$, such that 
\begin{align*} 
\|\varrho_{k_n}-\varphi_{k_n}\|_{\widetilde H_0^s(\Om)}^2 \leq C_{k_n}(s).
\end{align*}
Let the operator $L$ be defined by 
$$L:  \widetilde H_0^{s}(\Om)\to \widetilde H^{-s}(\mathcal O),\; v\mapsto Lv:=((-\partial_x^2)^sv)|_{\mathcal O}=(\mathcal N_sv)|_{\mathcal O}.$$

By \cite[Lemma 2.2]{GRSU},  the operator $L$ is compact, injective with dense range. Let $B_1:=\overline{B}\left(\varrho_{k_n},C_{k_n}(s)\right)$ be the closed ball in $\widetilde H_0^s(\Om)$ with center in $\varrho_{k_n}$ and radius $C_{k_n}(s)$. Since $L$ is a compact operator, we have that the image of $B_1$, namely $L(B_1)$, is totally bounded in $\widetilde H^{-s}(\mathcal{O})$. Therefore, for every $\varepsilon>0$ there exists $N\in\N$ and $\{\psi_1,\ldots,\psi_N\}\subseteq B_1$ such that
\begin{align*}
L(B_1)\subseteq \bigcup_{j=1}^{N}\overline{B}_{\widetilde H^{-s}(\mathcal O)}(L(\psi_j),\varepsilon).
\end{align*}
We notice that $\varphi_{k_n}$ belongs to $B_1$. Thus, there exists $j\in\{1,\ldots,N\}$ such that
\begin{align*}
L(\varphi_{k_n})\in \overline{B}_{\widetilde H^{-s}(\mathcal O)}(L(\psi_j),\varepsilon).
\end{align*}
We have shown that for $n$ large enough,
\begin{align*}
\|L(\varphi_{k_n})-L(\psi_j)\|_{\widetilde H^{-s}(\mathcal{O})}\leq \varepsilon.
\end{align*}

Since $\psi_j\in B_1$, firstly we obtain that $\varphi_{k_n}\rightarrow \psi_j$, as $n\to\infty$ in $\widetilde H_0^s(\Om)$  and secondly, we have that $\psi_j$ is an element of the spectrum $\{(\varphi_k,\lambda_k)\}_{k\geq 1}$. That is, $\psi_j$ is a solution of \eqref{ei-val-pro}. Finally, as $L(\varphi_{k_n})$ converges to zero in $\widetilde H^{-s}(\mathcal{O})$ (by \eqref{contr-3}), we can deduce that $L(\psi_j)=\mathcal N_s\psi_j=(-\partial_x^2)^s\psi_j=0$ a.e. in $\mathcal O$. It follows from Lemma \ref{lem-UCD} that $\psi_j=0$  a.e. in $\RR$, which is a contradiction. The proof of is finished.
\end{proof}

Now we can state and prove the  main result of this section.

\begin{theorem}\label{main2}
Let $\mathcal O\subset (\Omc)$ be an arbitrary nonempty open set. Then, for every $u_0\in L^2(\Om)$, $\frac 12<s<1$ and $T>0$, there exists a control function $g\in L^\infty(\mathcal O\times (0,T))\cap L^2((0,T);\widetilde H_0^s(\mathcal O))$ such that the corresponding unique weak solution $u$ of \eqref{eq-main} satisfies $u(x,T)=0$ for a.e. $x\in (-1,1)$. In addition, there is a constant $C=C(T)>0$ such that
\begin{align}\label{jjj}
 \|g\|_{L^\infty(\mathcal O\times (0,T))}\le\|g\|_{L^\infty(\mathcal O\times (0,T))\cap L^2((0,T);\widetilde H_0^s(\mathcal O))}\le C\|u_0\|_{L^2(-1,1)}.
\end{align}
\end{theorem}

\begin{proof}
Recall that by Theorem \ref{main1}, the null controllability of \eqref{eq-main} together with \eqref{jjj}, is equivalent to the $L^1$-observability inequality \eqref{observ}. Therefore, we shall prove that \eqref{observ} holds.

Let $T>0$, $\psi_T\in L^2(-1,1)$ and let $\psi\in C([0,T];L^2(-1,1))$ be the associated unique weak solution of \eqref{Dual}.
It follows from Theorem \ref{theo-48} that
\begin{align*} 
\psi(x,t)=\sum_{n=1}^{\infty}\psi_{0,n}e^{-\lambda_n(T-t)}\varphi_n(x)\;\;\;\mbox{ and }\;\;\;
\mathcal{N}_s\psi(x,t)=\sum_{n=1}^{\infty}\psi_{0,n}e^{-\lambda_n(T-t)}\mathcal{N}_s\varphi_n(x),
\end{align*}
where we recall that $\psi_{0,n}:=(\psi_T,\varphi_n)_{L^2(-1,1)}$.
Using the fact that $(\varphi_n)_{n\ge 1}$ is an orthonormal basis in $L^2(-1,1)$, we have that the $L^1$-observability inequality \eqref{observ} becomes
\begin{align}\label{OI1}
\sum_{n=1}^{\infty}|\psi_{0,n}|^2 e^{-2\lambda_n T}\leq C(T)\left(\int_0^T\int_{\mathcal{O}}\left|\sum_{n=1}^{\infty}\psi_{0,n}e^{-\lambda_n (T-t)} \mathcal{N}_s\varphi_n(x)\right|dxdt\right)^2.
\end{align}
Using the change of variable $T-t\mapsto t$, we get from \eqref{OI1} that
\begin{align}\label{obs-ine-1}
\sum_{n=1}^{\infty}|\psi_{0,n}|^2 e^{-2\lambda_n T}\leq C(T)\left(\int_0^T\int_{\mathcal{O}}\left|\sum_{n=1}^{\infty}\psi_{0,n}e^{-\lambda_n t} \mathcal{N}_s\varphi_n(x)\right|dxdt\right)^2.
\end{align}

We observe that $(\lambda_n)_{n\in\N}$ are simple (since we have assume that $\frac 12<s<1$) and the following asymptotics hold (see e.g. \cite{kwasnicki2012eigenvalues}):
\begin{align}\label{lam}
\lambda_n=\left(\frac{n\pi}{2}-\frac{(2-2s)\pi}{8}\right)^{2s}+O\left(\frac{1}{n}\right)\;\text{ as }\, n\to\infty.
\end{align}

Therefore, letting $\mu_n:=\lambda_n$ we have that the conditions (1) and (2) in Theorem \ref{muntz1} are both satisfied. Thus, we can deduce that \eqref{muntz} holds with $c_n$ replaced with $\psi_{0,n}$. 

Now, by \cite[Section 8, page 28, Equation (8.i)]{schwartz1943etude} and \cite[Section 9, page 33, Theorem I]{schwartz1943etude}, we have that for almost every fixed $x\in\mathcal O$, there exists a constant $C(T)>0$ which is uniformly bounded away from $T=0$, such that
\begin{align}\label{new}
\sum_{n=1}^{\infty}|\psi_{0,n}\mathcal{N}_s\varphi_n(x)|e^{-\lambda_n T}\leq C(T)\int_0^T\left|\sum_{n=1}^{\infty}\psi_{0,n}\mathcal{N}_s\varphi_n(x)e^{-\lambda_{n} t}\right|dt.
\end{align}

By Lemma \ref{uni-bound},  $\|\mathcal{N}_s\varphi_n\|_{L^1(\mathcal{O})}\ge \eta>0$. Thus, integrating \eqref{new} over $\mathcal{O}$ and using \eqref{eq39} we can deduce that 
\begin{align}\label{mw1}
 \eta \sum_{n=1}^{\infty}|\psi_{0,n}|e^{-\lambda_n T}\leq C(T)\int_{\mathcal{O}}\int_0^T\left|\sum_{n=1}^{\infty}\psi_{0,n}\mathcal{N}_s\varphi_n(x)e^{-\lambda_{n} t}\right|\;dtdx.
\end{align}
Since
\begin{align*}
\sum_{n=1}^{\infty}|\psi_{0,n}|^2e^{-2\lambda_n T}\leq \left(\sum_{n=1}^{\infty}|\psi_{0,n}|e^{-\lambda_n T}\right)^2,
\end{align*}
it follows from \eqref{mw1} that
\begin{multline*}
\eta^2\sum_{n=1}^{\infty}|\psi_{0,n}|^2e^{-2\lambda_n T}  \leq\eta^2 \left(\sum_{n=1}^{\infty}|\psi_{0,n}|e^{-\lambda_n T}\right)^2\\\leq C(T)^2\left(\int_{\mathcal{O}}\int_0^T\left|\sum_{n=1}^{\infty}\psi_{0,n}\mathcal{N}_s\varphi_n(x)e^{-\lambda_{n} t}\right|\;dtdx\right)^2.
\end{multline*}
Finally, using Fubini's theorem we get that
\begin{align*}
\sum_{n=1}^{\infty}|\psi_{0,n}|^2e^{-2\lambda_n T}\leq \frac{C(T)^2}{\eta^2}\left(\int_0^T\int_{\mathcal{O}}\left|\sum_{n=1}^{\infty}\psi_{0,n}e^{-\lambda_n t} \mathcal{N}_s\varphi_n(x)\right|dxdt\right)^2.
\end{align*}
We have shown that the $L^1$-observability inequality \eqref{observ} holds. The proof is finished.
\end{proof}

We conclude this section with the following observation.

\begin{remark}
{\em We mention the following facts.
\begin{enumerate}
\item We observe that since the constant $C(T)$ in \eqref{muntz} blows up exponentially as $T\downarrow 0^+$, we have that the constant in the $L^1$-obervability inequality \eqref{observ} also blows up exponentially as $T\downarrow 0^+$. This is consistent with the classical local case $s=1$, where the same phenomena occurs.
\item We mention that in this section we do not need the assumption that the control region $\mathcal O$ is bounded. This is due to the fact that we did not impose any constraints on the control function.
\item If $0<s\le \frac 12$, then the eigenvalues $(\lambda_n)_{n\ge 1}$ do not satisfy the conditions (1) and (2) in Theorem \ref{muntz1}. Thus, in this case, the null-controllability result in Theorem \ref{main2} does not hold.
\end{enumerate}
}
\end{remark}

\section{Proofs of the main results}\label{prof-ma-re}

In this section we give the proofs of the main results stated in Section \ref{main-results}.

\begin{proof}[\bf Proof of Theorem \ref{main3}]
Due to the linearity of \eqref{eq-main}, and considering $z:=u-\widehat{u}$ a solution of 
\begin{align}\label{eq-main1}
\begin{cases}
\partial_t z + (-\partial_x^2)^{s} z = 0 & \mbox{ in }\; (-1,1)\times(0,T),\\
z=h\chi_{\mathcal{O}\times(0,T)} &\mbox{ in }\; (\Omc)\times (0,T), \\
z(\cdot,0) = u_0-\widehat{u}_0&\mbox{ in }\; (-1,1),
\end{cases}
\end{align}
with $h:=g-\widehat{g}$, it is enough to prove that there exist $T>0$ and a control $h\in L^{\infty}(\mathcal{O}\times (0,T))\cap L^2((0,T);\widetilde H_0^s(\mathcal{O}))$ fulfilling $h\ge -\alpha$ a.e. in $\mathcal O\times(0,T)$ such that $z(\cdot,T)=0$ a.e. in $(-1,1)$. 

By Theorem \ref{main2}, the null controllability of \eqref{eq-main1} with $h\in L^{\infty}(\mathcal{O}\times (0,T))\cap L^2((0,T);\widetilde H_0^s(\mathcal{O}))$ is equivalent to \eqref{observ}. We observe that the $L^1$-observability inequality \eqref{observ} is independent of the time interval. For that reason we can also consider the interval $(t_0,T)$, for $t_0\in(0,T)$. Therefore, the $L^1$-observability inequality \eqref{observ} becomes
\begin{align}\label{observ1}
\|\psi(\cdot,0)\|_{L^2(-1,1)}^2\leq C(T-t_0)\left(\int_{t_0}^T\int_{\mathcal{O}}|\mathcal{N}_s\psi(x,t)|dxdt \right)^2.
\end{align}
It follows from \eqref{eq-25} that
\begin{align}\label{eq43}
\|\psi(\cdot,0)\|_{L^2(-1,1)}^2&\leq \sum_{n=1}^{\infty}|\psi_{0,n}|^2e^{-2\lambda_n T}|\varphi_n(x)|^2\nonumber \\
&=\sum_{n=1}^{\infty}|\psi_{0,n}|^2e^{-2\lambda_n (T-t_0)}e^{-2\lambda_n t_0}|\varphi_n(x)|^2,
\end{align}
where $\psi_{n,0}:=(\psi_T,\varphi_n)_{L^2(-1,1)}$.
Since $0<\lambda_1\leq\lambda_2\leq\ldots\leq\lambda_n\leq\ldots$, it follows from \eqref{eq43} that
\begin{align}\label{t1}
\|\psi(\cdot,0)\|_{L^2(-1,1)}^2&\leq e^{-2\lambda_1 t_0}\sum_{n=1}^{\infty}|\psi_{0,n}|^2e^{-2\lambda_n (T-t_0)}|\varphi_n(x)|^2=e^{-2\lambda_1 t_0}\|\psi(\cdot,t_0)\|_{L^2(-1,1)}^2.
\end{align}
Substituting \eqref{t1} into \eqref{observ1} we get that
\begin{align}\label{observ2}
\|\psi(\cdot,0)\|_{L^2(-1,1)}^2\leq e^{-2\lambda_1 t_0}C(T-t_0)\left(\int_{t_0}^T\int_{\mathcal{O}}|\mathcal{N}_s\psi(x,t)|dxdt \right)^2.
\end{align}

By Theorem \ref{main1}, \eqref{observ2} is equivalent to the existence of $h\in L^{\infty}(\mathcal{O}\times(0,T))\cap L^2((0,T);\widetilde H_0^s(\mathcal{O}))$ such that
\begin{multline}\label{eq46}
\|h\|_{L^{\infty}(\mathcal{O}\times(0,T))}^2\leq\|h\|_{L^{\infty}(\mathcal{O}\times(0,T))\cap L^2((0,T);\widetilde H_0^s(\mathcal{O}))}\\\leq e^{-2\lambda_1 t_0}C(T-t_0)\|u_0-\widehat{u}_0\|_{L^2(-1,1)}^2.
\end{multline}

Taking $t_0:=\frac{T}{2}$ and using the fact that the $L^1$-observability constant $C(T)$ is uniformly bounded away from $T=0$, we can deduce from \eqref{eq46} that for $T$ large enough, 
\begin{align}\label{eq47}
\|h\|_{L^{\infty}(\mathcal{O}\times(0,T))}^2\leq \alpha^2.
\end{align}

The estimate \eqref{eq47} implies that $h\ge -\alpha$ a.e. in $\mathcal O\times (0,T)$. We have constructed an exterior control $h\in L^{\infty}(\mathcal{O}\times(0,T))\cap L^2((0,T);\widetilde  H_0^s(\mathcal{O}))$ fulfilling the constraint $h\ge -\alpha$ a.e. in $\mathcal O\times(0,T)$, and is such that the solution $z$ of \eqref{eq-main1} satisfies $z(\cdot,T)=0$ a.e. in $(-1,1)$ for $T$ large enough. 
If $u_0\geq 0$, then from Theorem \ref{theo-28}, we have that $u\geq 0$ a.e. in $(-1,1)\times(0,T)$. The proof is finished.
\end{proof}

\begin{remark}
{\em For the controllabilty to trajectories result in Theorem \ref{main3} to hold, the control time $T$ must be large enough. This is due to the positivity constraints imposed on the control function. 
}
\end{remark}

\begin{proof}[\bf Proof of Theorem \ref{main4}]
Recall that by \eqref{4} the weak solution $u$ of \eqref{eq-main} is given by
\begin{align}\label{y1}
u(x,t)=\sum_{n=1}^{\infty}u_{0,n}e^{-\lambda_n t}\varphi_n(x)+\sum_{n=1}^{\infty}\left(\int_0^t (g(\cdot,\tau), \mathcal{N}_s\varphi_n)_{L^2(\mathcal{O})}  e^{-\lambda_n(t-\tau)}d\tau\right)\varphi_n(x).
\end{align}
Letting $u_n(t):=(u(\cdot,t),\varphi_n)_{L^2(-,1,1)}$, we get that
\begin{align}\label{y2}
u_n(t)=u_{0,n}e^{-\lambda_n t}+\int_0^t (g(\cdot,\tau), \mathcal{N}_s\varphi_n)_{L^2(\mathcal{O})}  e^{-\lambda_n(t-\tau)}d\tau.
\end{align}
Since $u(\cdot,T)=\widehat{u}(\cdot,T)$ a.e. in $(-1,1)$, it follows that
\begin{align}\label{y3}
u_n(T)=(\widehat{u}(\cdot,T),\varphi_n)_{L^2(-1,1)}=:z_n.
\end{align}
Substituting \eqref{y3} into \eqref{y2} we get that
\begin{align}\label{y4}
z_n-u_{0,n}e^{-\lambda_n T}=\int_0^T (g(\cdot,\tau), \mathcal{N}_s\varphi_n)_{L^2(\mathcal{O})}  e^{-\lambda_n(T-\tau)}d\tau.
\end{align}
We notice that
\begin{align*}
(g(\cdot,\tau), \mathcal{N}_s\varphi_n)_{L^2(\mathcal{O})}=(g(\cdot,\tau), [\mathcal{N}_s\varphi_n]^+)_{L^2(\mathcal{O})}-(g(\cdot,\tau), [\mathcal{N}_s\varphi_n]^-)_{L^2(\mathcal{O})},
\end{align*}
where for $v\in L^2(\mathcal O)$, we have set $v^+:=\sup\{v,0\}$ and $v^-:=\sup\{-v,0\}$.
Since
\begin{align*}
e^{-\lambda_nT}\leq e^{-\lambda_n(T-t_0)}\leq 1,\;\forall\;t_0\in[0,T]
\end{align*}
and $g(\cdot,\tau)\geq0$ a.e. in $\mathcal O$, we have that
\begin{align}\label{y5}
e^{-\lambda_nT}\int_0^T (g(\cdot,\tau), [\mathcal{N}_s\varphi_n]^+)_{L^2(\mathcal{O})}d\tau &\leq \int_0^T e^{-\lambda_n(T-\tau)} (g(\cdot,\tau), [\mathcal{N}_s\varphi_n]^+)_{L^2(\mathcal{O})}d\tau\notag\\
&\leq \int_0^T (g(\cdot,\tau), [\mathcal{N}_s\varphi_n]^+)_{L^2(\mathcal{O})}d\tau,
\end{align}
and
\begin{align}\label{y6}
e^{-\lambda_nT}\int_0^T (g(\cdot,\tau), [\mathcal{N}_s\varphi_n]^-)_{L^2(\mathcal{O})}d\tau &\leq \int_0^T e^{-\lambda_n(T-\tau)} (g(\cdot,\tau), [\mathcal{N}_s\varphi_n]^-)_{L^2(\mathcal{O})}d\tau\notag\\
&\leq \int_0^T (g(\cdot,\tau), [\mathcal{N}_s\varphi_n]^-)_{L^2(\mathcal{O})}d\tau.
\end{align}
From \eqref{y4} we have that
\begin{align}\label{yy5}
z_n-u_{0,n}e^{-\lambda_n T}+&\int_0^T  e^{-\lambda_n(T-\tau)} (g(\cdot,\tau), [\mathcal{N}_s\varphi_n]^-)_{L^2(\mathcal{O})}\;d\tau\notag\\
=&\int_0^T  e^{-\lambda_n(T-\tau)} (g(\cdot,\tau), [\mathcal{N}_s\varphi_n]^+)_{L^2(\mathcal{O})}\;d\tau,
\end{align}
and
\begin{align}\label{yy6}
z_n-u_{0,n}e^{-\lambda_n T}-&\int_0^T  e^{-\lambda_n(T-\tau)} (g(\cdot,\tau), [\mathcal{N}_s\varphi_n]^+)_{L^2(\mathcal{O})}\;d\tau\notag\\
=&-\int_0^T  e^{-\lambda_n(T-\tau)} (g(\cdot,\tau), [\mathcal{N}_s\varphi_n]^-)_{L^2(\mathcal{O})}\;d\tau.
\end{align}
Using \eqref{y5} and \eqref{yy5}, we get that
\begin{multline}\label{yyy5}
z_n-u_{0,n}e^{-\lambda_n T}+\int_0^T  e^{-\lambda_n(T-\tau)} (g(\cdot,\tau), [\mathcal{N}_s\varphi_n]^-)_{L^2(\mathcal{O})}\;d\tau\\
\le \int_0^T(g(\cdot,\tau), [\mathcal{N}_s\varphi_n]^+)_{L^2(\mathcal{O})}\;d\tau\\\le z_ne^{-\lambda_nT}-u_{0,n}+ \int_0^T  e^{\lambda_n\tau} (g(\cdot,\tau), [\mathcal{N}_s\varphi_n]^-)_{L^2(\mathcal{O})}\;d\tau.
\end{multline}
From \eqref{y6} and \eqref{yy6} we can deduce that
\begin{multline}\label{yyy6}
z_ne^{\lambda_nT}-u_{0,n}-\int_0^T  e^{\lambda_n\tau} (g(\cdot,\tau), [\mathcal{N}_s\varphi_n]^+)_{L^2(\mathcal{O})}\;d\tau\\
\le -\int_0^T(g(\cdot,\tau), [\mathcal{N}_s\varphi_n]^-)_{L^2(\mathcal{O})}\;d\tau\\\le z_n-u_{0,n}e^{-\lambda_nT}- \int_0^T  e^{\lambda_n\tau} (g(\cdot,\tau), [\mathcal{N}_s\varphi_n]^+)_{L^2(\mathcal{O})}\;d\tau.
\end{multline}

Now assume by contradiction that, for every $T>0$, there exists a non-negative exterior control $g^T$ steering $u_0$ to $\widehat{u}(\cdot,T)$ in time $T$, and that $\widehat{u}(\cdot,T)\neq u_0$ (otherwise the trival thajectory $u\equiv u_0\equiv \widehat{u}$ solves the problem). Then, applying \eqref{yyy5} with $g(\cdot,\tau):=g^T(\cdot,\tau)$ and taking the limit as $T\downarrow 0^+$, we get that
\begin{align}\label{y5-5}
\lim_{T\downarrow 0^+}\int_0^T(g(\cdot,\tau), [\mathcal{N}_s\varphi_n]^+)_{L^2(\mathcal{O})}\;d\tau=z_n-u_{0,n}.
\end{align}

Similarly, applying \eqref{yyy6} with $g(\cdot,\tau):=g^T(\cdot,\tau)$ and taking the limit as $T\downarrow 0^+$, we get that
\begin{align}\label{y6-6}
\lim_{T\downarrow 0^+}-\int_0^T(g(\cdot,\tau), [\mathcal{N}_s\varphi_n]^-)_{L^2(\mathcal{O})}\;d\tau=z_n-u_{0,n}.
\end{align}
It follows from \eqref{y5-5} and \eqref{y6-6} that
\begin{align}
\lim_{T\downarrow 0^+}\int_0^T&(g(\cdot,\tau), \mathcal{N}_s\varphi_n)_{L^2(\mathcal{O})}\;d\tau\notag\\
=& \lim_{T\downarrow 0^+}\int_0^T(g(\cdot,\tau), [\mathcal{N}_s\varphi_n]^+)_{L^2(\mathcal{O})}\;d\tau
-\lim_{T\downarrow 0^+}\int_0^T(g(\cdot,\tau), [\mathcal{N}_s\varphi_n]^-)_{L^2(\mathcal{O})}\;d\tau\notag\\
=&\,2(z_n-u_{0,n})=:\gamma.
\end{align}
Since $u_0\in L^2(\Om)$, we have that
\begin{align*}
\sum_{n=1}^\infty|u_{0,n}|^2=\sum_{n=1}^\infty \left(z_n^2-z_n\gamma+\frac{\gamma^2}{4}\right)<\infty,
\end{align*}
which implies that
\begin{align}\label{mj}
\lim_{n\to\infty}\left(z_n^2-z_n\gamma+\frac{\gamma^2}{4}\right)=0.
\end{align}

Since $(\varphi_n)_{n\ge 1}$ is an orthonormal complete system in $L^2(-1,1)$, we have that $\varphi_n\rightharpoonup 0$ (weak convergence) in $L^2(-1,1)$ as $n\to\infty$. This implies that 
\begin{align*}
\lim_{n\to\infty}z_n=\lim_{n\to\infty}(\widehat{u}(\cdot,T),\varphi_n)_{L^2(-1,1)}=0.
\end{align*}
The above convergence together with \eqref{mj} yield $\gamma=0$. We have then shown that
\begin{align*}
0=2(z_n-u_{0,n})=2\int_{-1}^1\big(\widehat{u}(x,T)-u_0(x)\Big)\varphi_n(x)\;dx,\;\forall\;n\ge 1.
\end{align*}

This is possible if and only if $u_0(x)=\widehat{u}(x,T)$ for a.e. $x\in (-1,1)$, which is a contradiction to our assumption. The proof is finished.
\end{proof}

Before we proceed with the proof of our last main result, we need some preparations. 


\begin{lemma}
Let $\mathcal O\subset\overline{\mathcal O}\subset(\R \setminus[-1,1])$ be an arbitrary nonempty bounded open set. Then, there are two constants $0<C_1\le C_2$ such that for every $x\in\mathcal O$, we have
\begin{align}\label{C1}
C_1\le \int_{-1}^1\frac{dy}{|x-y|^{1+2s}}\le C_2.
\end{align}
\end{lemma}

\begin{proof}
Since $\mathcal O\subset\overline{\mathcal O}\subset(\R \setminus[-1,1])$, 
we have that there are two constants $1< a\le b$ such that $1<a\le |x|\le b$ for every $x\in\mathcal O$.  Thus, we have the following two cases.
\begin{itemize}
\item Case 1:  $1<a\le x\le b$. A simple calculation gives
\begin{align*}
\int_{-1}^1\frac{dy}{|x-y|^{1+2s}}=\frac{1}{2s}\left(\frac{1}{(x-1)^{2s}}-\frac{1}{(x+1)^{2s}}\right).
\end{align*}
Define $f:[a,b]\to [0,\infty)$ by $f(x):=\frac{1}{2s}\left(\frac{1}{(x-1)^{2s}}-\frac{1}{(x+1)^{2s}}\right)$. Then, $f$ is decreasing. Thus 
\begin{align}\label{f}
f(b)\le f(x)\le f(a)\;\mbox{ for every }\; a\le x\le b.
\end{align}

\item Case 2: $-b\le x\le -a<-1$. Then
\begin{align*}
\int_{-1}^1\frac{dy}{|x-y|^{1+2s}}=\frac{1}{2s}\left(\frac{1}{(-1-x)^{2s}}-\frac{1}{(1-x)^{2s}}\right).
\end{align*}
Define $\tilde f:[-b,-a]\to [0,\infty)$ by $\tilde f(x):=\frac{1}{2s}\Big(\frac{1}{(-1-x)^{2s}}-\frac{1}{(1-x)^{2s}}\Big)$. Then, $\tilde f$ is increasing. Thus 
\begin{align}\label{ff}
\tilde f(-b)\le f(x)\le \tilde f(-a)\;\mbox{ for every }-b\le x\le -a.
\end{align}
\end{itemize}
Now \eqref{C1} follows from \eqref{f} and \eqref{ff}. The proof is finished.
\end{proof}


Next, we recall that the non-local normal derivative of the solution $\psi$ to the adjoint system \eqref{Dual} is given by
\begin{align}\label{series}
\mathcal{N}_s \psi(x,t)=\sum_{n=1}^{\infty}\psi_{0,n}e^{-\lambda_n(T-t)}\mathcal{N}_s\varphi_{n}(x).
\end{align}
We have the following result.

\begin{lemma}
Let $\mathcal O\subset\overline{\mathcal O}\subset(\R \setminus[-1,1])$ be an arbitrary nonempty bounded open set. Let $\psi$ be the unique weak solution of the dual system \eqref{Dual}.
If $\psi_T\in L^{\infty}(-1,1)$, then $\mathcal{N}_s \psi\in L^{\infty}(\mathcal O\times(0,T))$.
\end{lemma}

\begin{proof}
Firstly, we claim that $\mathcal{N}_s \varphi_{n}\in L^{\infty}(\mathcal O)$ for every $n\in\NN$. Indeed, notice that the eigenfunction $\varphi_n\in L^\infty(-1,1)$ for every $n\in\NN$ and $\varphi_n(x)=0$ for a.e. $x\in\mathcal O$. Thus, for a.e. $x\in\mathcal O$ we have that
\begin{align}
|\mathcal{N}_s \varphi_{n}(x)|&\leq C_{s}\int_{-1}^1\left|\frac{\varphi_n(x)-\varphi_n(y)}{|x-y|^{1+2s}}\right|\;dy 
\leq C_{s}\int_{-1}^1\frac{|\varphi_n(y)|}{|x-y|^{1+2s}}\;dy\nonumber\\
&\leq C_{s} \|\varphi_n\|_{L^{\infty}(-1,1)}\int_{-1}^1\frac{1}{|x-y|^{1+2s}}\;dy \leq C_{s} C_2\|\varphi_n\|_{L^{\infty}(-1,1)},\label{acota}
\end{align} 
where in the last estimate we have used \eqref{C1}.  It follows from \eqref{acota} that $\mathcal{N}_s \varphi_{n}\in L^{\infty}(\mathcal O)$ for every $n\in\NN$. Now using \eqref{series}, we get that for a.e. $(x,t)\in \mathcal O\times (0,T)$,
\begin{align*}
|\mathcal{N}_s \psi(x,t)|\le& \sum_{n=1}^{\infty}\left|\psi_{0,n}e^{-\lambda_n(T-t)}\mathcal{N}_s\varphi_{n}(x)\right|\\
\le& \|\psi_T\|_{L^\infty(-1,1)}\|\mathcal N_s\varphi_n\|_{L^\infty(\mathcal O)}\sum_{n=1}^\infty e^{-\lambda_n(T-t)}\\
\le& C(T)\|\psi_T\|_{L^\infty(-1,1)}\|\mathcal N_s\varphi_n\|_{L^\infty(\mathcal O)}<\infty.
\end{align*}
The proof is finished.
\end{proof}

We recall that $\mathcal{M}(\mathcal{O}\times(0,T))$
 is the space of Radon measures endowed with the norm 
\begin{multline*}
\|\mu\|_{\mathcal{M}(\mathcal{O}\times(0,T))}:=\sup\Big\{\int_{\mathcal{O}\times(0,T)} \xi(x,t) d\mu(x,t)\; : \; \xi\in C_c(\overline{\mathcal{O}}\times[0,T],\RR),\\\; \max_{\overline{\mathcal{O}}\times[0,T]}|\xi|=1  \Big\}.
\end{multline*}

Next, we introduce our notion of solutions to the system \eqref{eq-main} with an exterior measure datum. 

\begin{definition}
Let $u_0\in L^2(-1,1)$, $T>0$ and $g\in \mathcal{M}(\mathcal{O}\times(0,T))$. We shall say that the function $u\in L^1((-1,1)\times(0,T))$ is a solution of \eqref{eq-main} defined by transposition, if it satisfies the identity
\begin{align}\label{transposition}
\int_{\mathcal{O}\times(0,T)}\mathcal{N}_s\psi(x,t)dg(x,t)=\int_{-1}^1 u_0(x)\psi(x,0)dx-\langle u(\cdot,T), \psi_T\rangle_{L^1(-1,1),L^\infty(-1,1)},
\end{align}
where for every $\psi_T\in L^{\infty}(-1,1)$, $\psi\in L^{\infty}((-1,1)\times(0,T))$ is the unique weak solution of 
\begin{equation}\label{Dual trans}
\begin{cases}
-\partial_t \psi +(-\partial_x^2)^s\psi=0\;\;&\mbox{ in }\; (-1,1)\times (0,T),\\
\psi=0&\mbox{ in }\;(\Omc)\times (0,T),\\
\psi(\cdot,T)=\psi_T&\mbox{ in }\;(-1,1).
\end{cases}
\end{equation} 
\end{definition}

Now we are ready to give the proof of the last main result.

\begin{proof}[\bf Proof of Theorem \ref{radon measure}]
By definition of the minimal controllability time $T_{\min}$, we have that for each 
\begin{align*}
T_k:=T_{\min}+\frac{1}{k}, \quad k\geq 1,
\end{align*}
there exists a sequence of non--negative controls  
\begin{align*}
	(g^{T_k})_{k\geq 1}\subset L^{\infty}(\mathcal{O}\times (0,T_k))\cap L^2((0,T_k);\widetilde H_0^s(\mathcal O))
\end{align*}
such that the associated solutions $(u^k)_{k\ge 1}$ of \eqref{eq-main} with initial data $u^k(\cdot,0)=u_0$ a.e. in $(-1,1)$, satisfy $u^k(x,T_k)=\widehat{u}(x,T_k)$ for a.e. $x\in (-1,1)$. We extend these controls by $\widehat{g}$ in $(T_k, T_{\min}+1)$ to get a new sequence of controls $\{g^{T_k}\}_{k\geq 1}\subset  L^{\infty}(\mathcal{O}\times(0,T_{\min}+1))\cap L^2((0,T_{\min}+1);\widetilde H_0^s(\mathcal{O}))$.

Let $\varphi_1$ be the first non-negative eigenfunction of $(-\partial_x^2)_D^s$ (see \eqref{ei-val-pro}) and consider the  problem
\begin{equation}\label{Dual-eigen}
\begin{cases}
-\partial_t \psi +(-\partial_x^2)^s\psi=0\;\;&\mbox{ in }\; (-1,1)\times (0,T_{\min}+1),\\
\psi=0&\mbox{ in }\;(\Omc)\times (0,T_{\min}+1),\\
\psi(\cdot,T_{\min}+1)=\varphi_1&\mbox{ in }\;(-1,1).
\end{cases}
\end{equation} 

Firstly, the solution $\psi$ of \eqref{Dual-eigen} satisfies $\psi\in C([0,T_{\min}+1];D((-\partial_x^2)_D^s))\hookrightarrow C([-1,1]\times[0,T_{\min}+1])$. Secondly, due to Theorem \ref{theo-28} we have that there is a constant $\alpha>0$ such that
\begin{align}\label{aa}
\psi(x,t)\geq \alpha>0 \quad \forall (x,t)\in (-1,1)\times(0,T_{\min}+1).
\end{align}
Besides, using \eqref{C1} and \eqref{aa}, we get that for a.e. $(x,t)\in \mathcal O\times (0,T)$, 
\begin{align*}
\mathcal{N}_s\psi(x,t)=C_s\int_{-1}^1\frac{-\psi(y,t)}{|x-y|^{1+2s}}dy\leq -C_s\alpha\int_{-1}^1\frac{1}{|x-y|^{1+2s}}dy\le -C_sC_1\alpha.
\end{align*}
Therefore, taking $\beta:=C_sC_1\alpha>0$, we get that
\begin{align*}
\mathcal{N}_s\psi(x,t)\leq -\beta, \quad \mbox{ for a.e. } (x,t)\in \mathcal O\times(0,T_{\min}+1).
\end{align*}

Using the positivity of $g^{T_k}$ and \eqref{transposition}, we we can deduce that there is a constant $M>0$ such that
\begin{align*}
\beta\|g^{T_k}\|_{L^1(\mathcal{O}\times(0,T_{\min}+1))}&=\beta\int_0^{T_{\min}+1}\int_\mathcal{O}g^{T_k}(x,t)dxdt\\
&\leq \int_0^{T_{\min}+1}\int_\mathcal{O}-\mathcal{N}_s\psi(x,t) \ g^{T_k}(x,t)dxdt\\
&=\langle u(\cdot,T_{\min}+1),\varphi_1\rangle_{L^1(-1,1),L^\infty(-1,1)} - \int_{-1}^1u_0(x)\psi(x,0)dx \\
&\leq M,
\end{align*}
where the last estimate follows from the continuous dependence of solutions on the initial data.
We have shown that the sequence $\{g^{T_k}\}_{k\geq 1}$ is bounded in $L^1(\mathcal{O}\times(0,T_{\min}+1))$, and hence, it is bounded in $\mathcal{M}(\mathcal{O}\times(0,T_{\min}+1))$. Thus, there exists $\widetilde{g}\in \mathcal{M}(\mathcal{O}\times(0,T_{\min}+1))$ such that, up to a subsequence if necessary, 
\begin{align*}
g^{T_k} \rightharpoonup \widetilde{g} \quad \text{weakly--}\star \text{ in }\mathcal{M}(\mathcal{O}\times(0,T_{\min}+1)), \quad \text{ as }k\to\infty.
\end{align*}
It is also clear that $\widetilde{g}$ satisfies the non-negativity constraint.

Next, for every $k$ large enough and $T_{\min}<T_0<T_{\min}+1$, using \eqref{transposition} and the fact that $g^{T_k}$ is a trajectory control, we get that for every $\psi_{T_0}\in L^{\infty}(-1,1)$,
\begin{align}\label{hh}
\int_0^{T_0}\int_{\mathcal{O}}\mathcal{N}_s\psi(x,t)dg^{T_k}(x,t)=\int_{-1}^1u_0(x)\psi(x,0)dx-\langle \widehat{u}(\cdot,T_0),\psi_{T_0}\rangle_{L^1(-1,1),L^\infty(-1,1)}.
\end{align}

In particular, taking $\psi_{T_0}$ smooth enough, we get that $\mathcal{N}_s\psi \in C((\overline{\mathcal{O}}\times[0,T])$. Thus, by the weak-$\star$ convergence, taking the limit of\eqref{hh} as $k\to\infty$, we get that
\begin{align}\label{hh1}
\int_0^{T_0}\int_{\mathcal{O}}\mathcal{N}_s\psi(x,t)d\widetilde{g}(x,t)=\int_{-1}^1u_0(x)\psi(x,0)dx-\langle \widehat{u}(\cdot,T_0),\psi_{T_0}\rangle.
\end{align}

The identity \eqref{hh1} together with \eqref{transposition} imply that $u(x,T_0)=\widehat{u}(x,T_0)$ for a.e. $x\in (-1,1)$. Finally, taking the limit as $T_0\to T_{\min}$ and using the fact that 
\begin{align*}
|\widetilde{g}|(\mathcal{O}\times(T_{\min},T_0))=|\widehat{g}|(\mathcal{O}\times(T_{\min},T_0))=0, \quad \text{ as }T_0\to T_{\min},
\end{align*}
we can deduce that $u(x,T_{\min})=\widehat{u}(x,T_{\min})$ for a.e. $x\in (-1,1)$. The proof is complete.
\end{proof}

\section{Numerical simulations}\label{sec-num}

Our main Theorems \ref{main3}, \ref{main4}, and \ref{radon measure} state that the non-local heat equation \eqref{eq-main} is controllable from every initial datum $u_0\in L^2(-1,1)$ to any positive trajectory $\widehat{u}$, by using a non-negative control $g\in L^{\infty}((0,T);L^2(\mathcal{O}))\cap L^2((0,T);\widetilde H_0^s(\mathcal{O}))$, whenever $\frac 12<s<1$, $\mathcal{O}\subset \overline{\mathcal O}\subset(\R\setminus[-1,1])$ is a bounded open set,  and the controllability time is large enough. Moreover, in the minimal controllability time $T_{\rm min}$, this same result is achieved with controls in the space of Radon measures.

The aim of this final section is to present some numerical examples confirming these theoretical conclusions. To this end, we shall first discuss how to approximate the following exterior problem:
\begin{align}\label{HE}
	\begin{cases}
		\partial_t u + (-\partial_x^2)^{s} u = 0 & \mbox{ in }\; (-1,1)\times(0,T),
		\\
		u=g &\mbox{ in }\; (\Omc)\times (0,T), 
		\\
		u(\cdot,0) = 0&\mbox{ in }\; (-1,1).
	\end{cases}
\end{align}

In what follows, we will employ a FE approach, which is based on the variational formulation associated to \eqref{HE}. Notice that \eqref{HE} is not the classical one-dimensional boundary problem, in which the non-homogeneous datum $g$ is supported on the boundary $\{-1\}\times(0,T)$ or $\{1\}\times(0,T)$. The fact that $g$ is supported in the exterior of the domain $(-1,1)$ introduces some difficulties in the approximation process which requires a more careful analysis.

We impose the exterior condition in \eqref{HE} by using the approach from \cite{antil2019external1} (see also \cite{antil2019external} for the stationary problem). We first approximate the Dirichlet problem \eqref{HE} by the fractional Robin problem
\begin{align}\label{HE-approx}
	\begin{cases}
		\partial_t u^n + (-\partial_x^2)^{s} u^n = 0 & \mbox{ in }\; (-1,1)\times(0,T),
		\\
		\mathcal{N}_su^n+n\kappa u^n=n\kappa g &\mbox{ in }\; (\Omc)\times (0,T), 
		\\
		u^n(\cdot,0) = u_0 & \mbox{ in }\; (-1,1),
	\end{cases}
\end{align}
where $n\in\NN$ is a fixed, $\kappa\in L^1(\Omc)\cap L^\infty(\Omc)$ is a given non-negative function. Indeed, it has been shown in the aforementioned reference that the weak solution $u^n$ to \eqref{HE-approx} converges to a weak solution $u$ to \eqref{HE}, at a rate of $\mathcal O(n^{-1})$. More precisely, if we let the solution space of $u^n$ to be
\begin{align*}
H_\kappa^s(-1,1):=\Big\{u:\R\to\R\;\mbox{ measurable and }\;\|u\|_{H_\kappa^s(-1,1)}<\infty\Big\},
\end{align*}
where
\begin{align*}
\|u\|_{H_\kappa^s(-1,1)}^2:=\int_{-1}^1|u|^2\;dx+\int_{\Omc}|u|^2\kappa\;dx +\int_{\R^2\setminus(\Omc)^2}\frac{|u(x)-u(y)|^2}{|x-y|^{1+2s}}\;dxdy , 
\end{align*}
then the following result holds (cf.~\cite[Theorem 5.3]{antil2019external1}).

\begin{theorem}\label{Th-approx}
Let $g\in H^1((0,T);H^s(\Omc))$ and $u^n\in L^2((0,T); H_\kappa^s(-1,1)\\\cap L^2(\Omc))\cap H^1((0,T);(H_\kappa^s(-1,1)\cap L^2(\Omc))^{*})$ be the weak solution of \eqref{HE-approx}. Let $u\in L^2((0,T);H^s(\R))\cap H^1((0,T); \widetilde H^{-s}(-1,1))$ be the weak solution of \eqref{HE}. Then, there is a constant $C>0$, independent of n, such that
\begin{align}\label{approximation}
	\|u-u^n\|_{L^2((0,T);L^2(\R))}\leq \frac{C}{n}\|u\|_{L^2((0,T);H^s(\R))}.
\end{align}
In particular, $u^n$ converges strongly to $u$ in $L^2((0,T); L^2(-1,1))=L^2((-1,1)\times(0,T))$ as $n\to\infty$. 
\end{theorem}

Thus for a sufficiently large $n$, \eqref{HE-approx} approximates \eqref{HE} well. In view of that, for the remainder of this section, instead of \eqref{HE} we will consider \eqref{HE-approx} with $n=10^9$, giving an approximation of the order $\mathcal O(10^{-9})$.

Concerning now the control problem, we discretize \eqref{HE-approx} in the interval $(-2,2)$ by assuming that the control function $g$ is supported in a subset $\mathcal O$ of $((-2,2)\setminus [-1,1])$. In that case, we can take  $\kappa=1$ and the control function $g$ to be supported in $\mathcal O\times(0,T)$ by multiplying it with the characteristic function $\chi_{\mathcal O\times(0,T)}$. In other words, we will consider the following control problem:
\begin{align}\label{main-num}
	\begin{cases}
		\partial_t u^n + (-\partial_x^2)^{s} u^n = 0 & \mbox{ in }\; (-1,1)\times(0,T),
		\\
		\mathcal{N}_su^n+nu^n=ng\chi_{\mathcal{O}\times(0,T)} &\mbox{ in }\; ((-2,2)\setminus(-1,1))\times (0,T), 
		\\
		u^n(\cdot,0) = u_0&\mbox{ in }\; (-1,1).
	\end{cases}
\end{align}
For the target trajectory, we consider
\begin{align}\label{eq:widehatu}
	\widehat{u}(x,T):=\frac{\Gamma\left(\frac 12\right)2^{-2s}e^T}{\Gamma(1+s)\Gamma\left(\frac 12 +s\right)}\left(1-|x|^2\right)_+^s,
\end{align}
which is known (see for instance \cite{antil2019external1}) to be the exact solution to the Dirichet problem evaluated at the final time $T$, i.e., $\widehat{u}$ satisfies
\begin{align}\label{main-num-1}
	\begin{cases}
		\partial_t \widehat{u} + (-\partial_x^2)^{s} \widehat{u} = z_{exact}+e^t & \mbox{ in }\; (-1,1)\times(0,1),
		\\
		\widehat{u}=z_{exact} &\mbox{ in }\; ((-2,2)\setminus(-1,1))\times (0,1), 
		\\
		\widehat{u}(\cdot,0) = z_{exact}(\cdot,0)&\mbox{ in }\; (-1,1),
\end{cases}
\end{align}
where 
\begin{align*}
	z_{exact}(x,t):=\frac{\Gamma\left(\frac 12\right)2^{-2s}e^t}{\Gamma(1+s)\Gamma\left(\frac 12 +s\right)}\left(1-|x|^2\right)_+^s.
\end{align*}
We focus on the following two specific situations:
\begin{itemize}
	\item \textbf{Case 1}: Set the initial datum to be
	\begin{align*}
		u_0(x) := \frac 12\cos\left(\frac \pi2 x\right).
	\end{align*}
	In this case, we have that $u_0<\widehat{u}(\cdot,T)$ in $(-1,1)$, where $\widehat{u}$ is as in \eqref{eq:widehatu}.
		
	\item \textbf{Case 2}: Set the initial datum to be
	\begin{align*}
		u_0(x) := 1.8\cos\left(\frac \pi2 x\right).
	\end{align*}
	In this case, we have that $u_0>\widehat{u}(\cdot,T)$ in $(-1,1)$, where $\widehat{u}$ is as in \eqref{eq:widehatu}.
\end{itemize}

In both cases, we first estimate numerically $T_{\rm min}$ by formulating the minimal-time control problem as an optimization problem. We show that in this computed minimal time, the fractional heat equation \eqref{eq-main} is controllable from $u_0\in L^2(-1,1)$ to the given trajectory $\widehat{u}(\cdot,T)$ (cf.~\eqref{eq:widehatu}) by means of a non-negative control $g$. Secondly, we will show that, for $T< T_{\rm min}$ this controllability result is not achieved.

In all cases, we choose the sub-interval $\mathcal O=(1.7,1.9)\subset ((-2,2)\setminus [-1,1])$ as the control region. Moreover, we focus on the case $\frac 12<s<1$, where we know that \eqref{eq-main} is controllable. In particular, we will always take $s=0.8$.

\subsection{Case 1: $u_0<\widehat{u}(\cdot,T)$}

We first consider the case where the initial datum $u_0$ is below the final target $\widehat u(\cdot,T)$. We begin by estimating the minimal controllability time $T_{\rm min}$ by solving an optimization problem. Next we address the numerical constrained controllability of \eqref{eq-main} in a time horizon $T\geq T_{\rm min}$. Finally, we  consider the case where $T< T_{\rm min}$.

\subsubsection{Calculation of minimal controllability time $T_{\rm min}$}

To obtain $T_{\rm min}$, we consider the following constrained optimization problem:
\begin{align}\label{minT}
	\text{minimize}\quad T
\end{align}
subject to
\begin{align}\label{min-opt}
	\begin{cases}
		T>0,
		\\
		\partial_t u^n + (-\partial_x^2)^{s} u^n = 0 & \mbox{ in }\; (-1,1)\times(0,T),
		\\
		\mathcal{N}_su^n+nu^n=ng\chi_{\mathcal{O}\times(0,T)} &\mbox{ in }\; ((-2,2)\setminus(-1,1))\times (0,T), 
		\\
		u^n(\cdot,0) = u_0\geq 0&\mbox{ in }\; (-1,1),
		\\
		g\geq 0 & \mbox{ in }\; \mathcal{O}\times(0,T) ,
	\end{cases}
\end{align}
which we solve using {\tt CasADi} open-source tool for nonlinear optimization and algorithmic differentiation \cite{Andersson2019}.
We stress that, in the above optimization problem, both $T$ and $g$ will be considered as variables which need to be computed.

The PDE in \eqref{min-opt} is discretized over a uniform partition of the space interval $(-2, 2)$ as follows:
\begin{align*}
	-2=x_0<x_1<\ldots<x_{N-1}<x_{N}=2,
\end{align*}
where $x_i=x_{i-1}+h$, for all $i\in\{0,1,\ldots,N\}$, with $h$ denoting the distance between two consecutive points. We use $\mathfrak{M}$ to denote a mesh with points $\{x_i:\; i=0,1,\ldots,N\}$. In all our examples we have set $N=210$. 

We use globally continuous piece-wise linear finite element method on the aforementioned mesh to discretize in space. We denote the resulting finite element space by $\V_h$. We apply Backward-Euler, on a grid $t_k = \frac{Tk}{N_t}$,  $k = 0,\ldots,M$, to discretize in time. In all our experiments, we have set $N=210$ and $M=300$. Then, given $u_h^0 = u_0$, for $k=1,\ldots,M$, we need to solve for $u_h^k\in \V_h$ via
\begin{align}\label{vartional-discrete}
	\int_{-1}^1\frac{u_h^{k}-u_h^{k-1}}{\delta t} v \ dx+ n\mathcal{F}(u_h^k, v)+ \int_{\Omc} nu_h^kv \ dx=\int_{\mathcal{O}} ng^kv \ dx , \quad \forall v\in \V_h,
\end{align}
where the closed bilinear form $\mathcal{F}$ is given in \eqref{closed}. The approximation of $\mathcal{F}(u_h^k, v)$ is carried out by using the approach of \cite{BiHe}.

By solving \eqref{minT} we obtain that $T_{\rm min} = 0.4739$. Next, we solve the state equation with $T = T_{\rm min}$, the results are given in Figure \ref{fig1}. We clearly notice that in this time horizon, we are able to steer the initial datum $u_0$ to the desired target $\widehat{u}$ while maintaining the positivity of the solution.

\begin{SCfigure}[0.45][h]
	\centering 
	\includegraphics[scale=0.25]{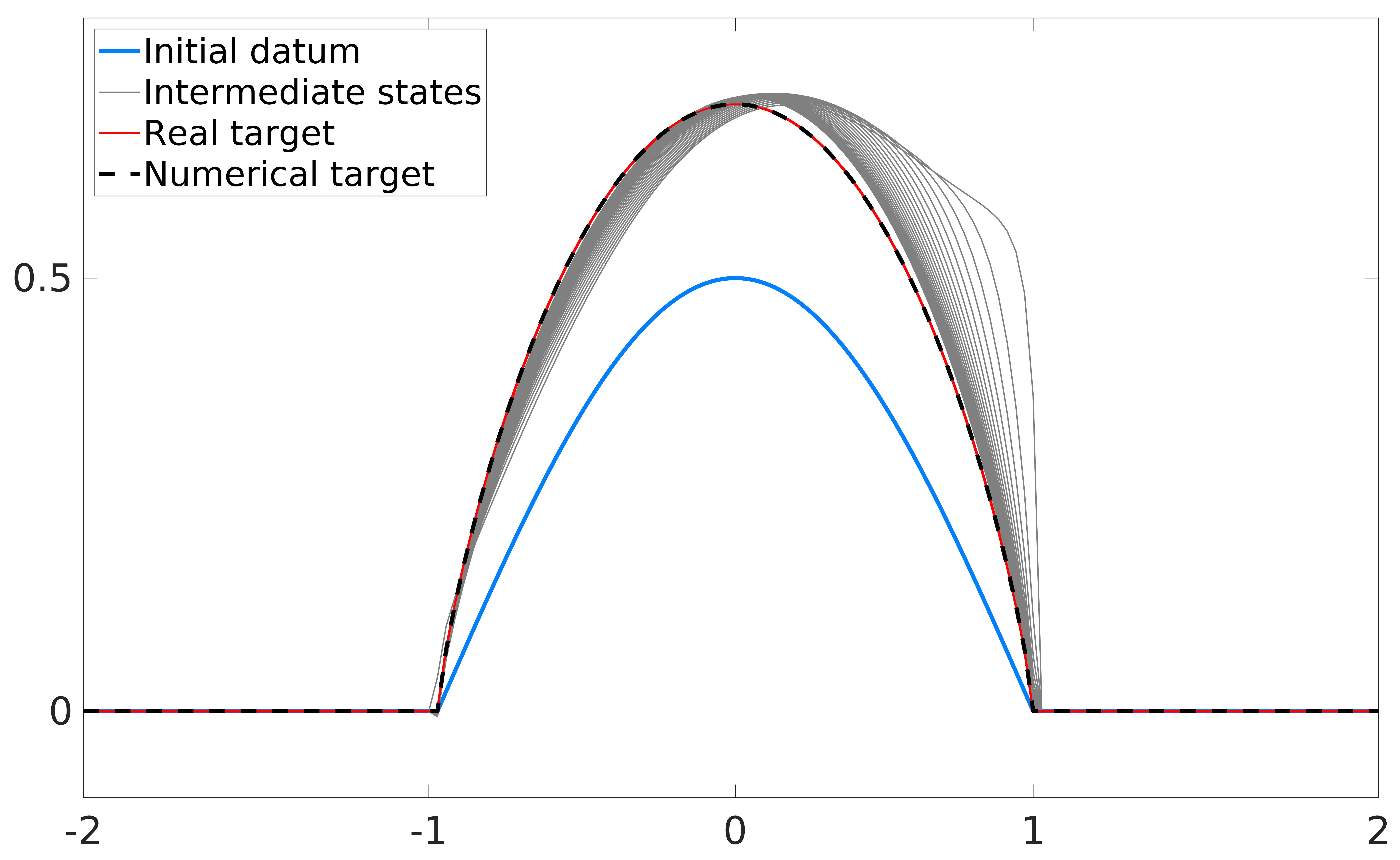}
	\caption{Evolution of the solution to \eqref{HE-approx} in the time interval $(0,T_{\min})$ with $s=0.8$. The blue line is the initial configuration $u_0$. The red line is the target $\widehat{u}(\cdot,T)$ ($T = T_{\rm min}$) configuration. The black dashed line is the numerical solution at $T = T_{\rm min}$}.
	\label{fig1}
\end{SCfigure}

The Figures \ref{fig2} and \ref{fig3} show the behavior of the control from $t = 0$ to $T = T_{\rm min}$.  Since the amplitude of control impulses is comparatively large, therefore, we have used logarithmic scale to plot Figure~\ref{fig3}.  We notice that at first, the control produces an initial shock and as a result it raises the value of the solution close to the final target. After an intermediate period, it shows an impulsive behavior to adjust to the trajectory of the desired state. Notice that the controllability at $T = T_{\rm min}$ and the impulsive behavior are both according to our theoretical results.

Intuitively the behavior of the control in Figures \ref{fig2} and \ref{fig3} is natural. Our goal is to reach a  target which is above the initial datum $u_0$. This means that the control needs to countervail the dissipation of the solution of \eqref{main-num}, by acting on it from the very beginning with a positive force. 

\begin{SCfigure}[0.7][h]
	\includegraphics[scale=0.45]{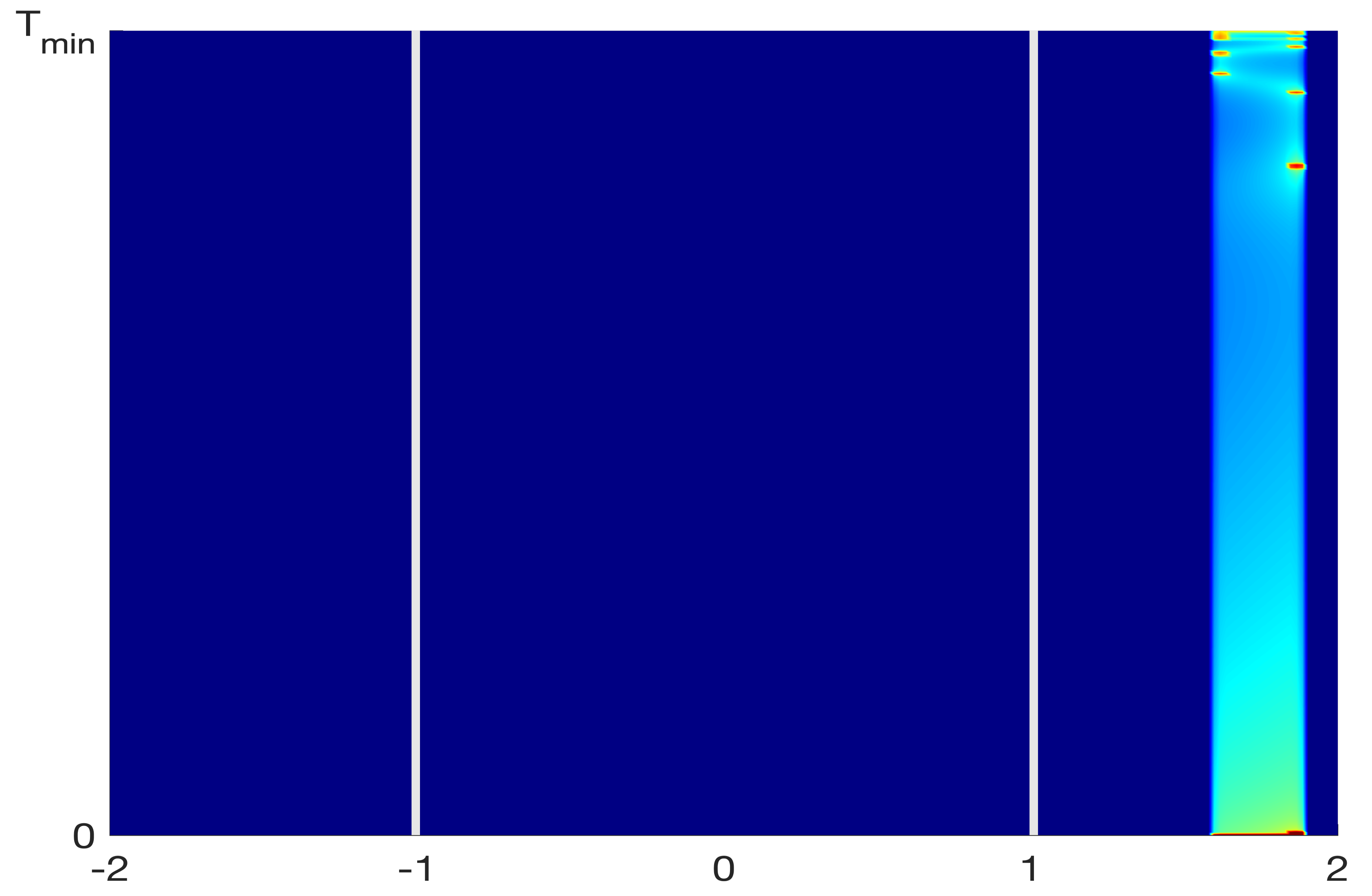}
	\caption{Minimal-time control: space-time distribution of the control. The white lines delimit the dynamics region $(-1,1)$.}
	\label{fig2}
\end{SCfigure}

\begin{SCfigure}[0.5][h]
	\centering 
	\includegraphics[scale=0.55]{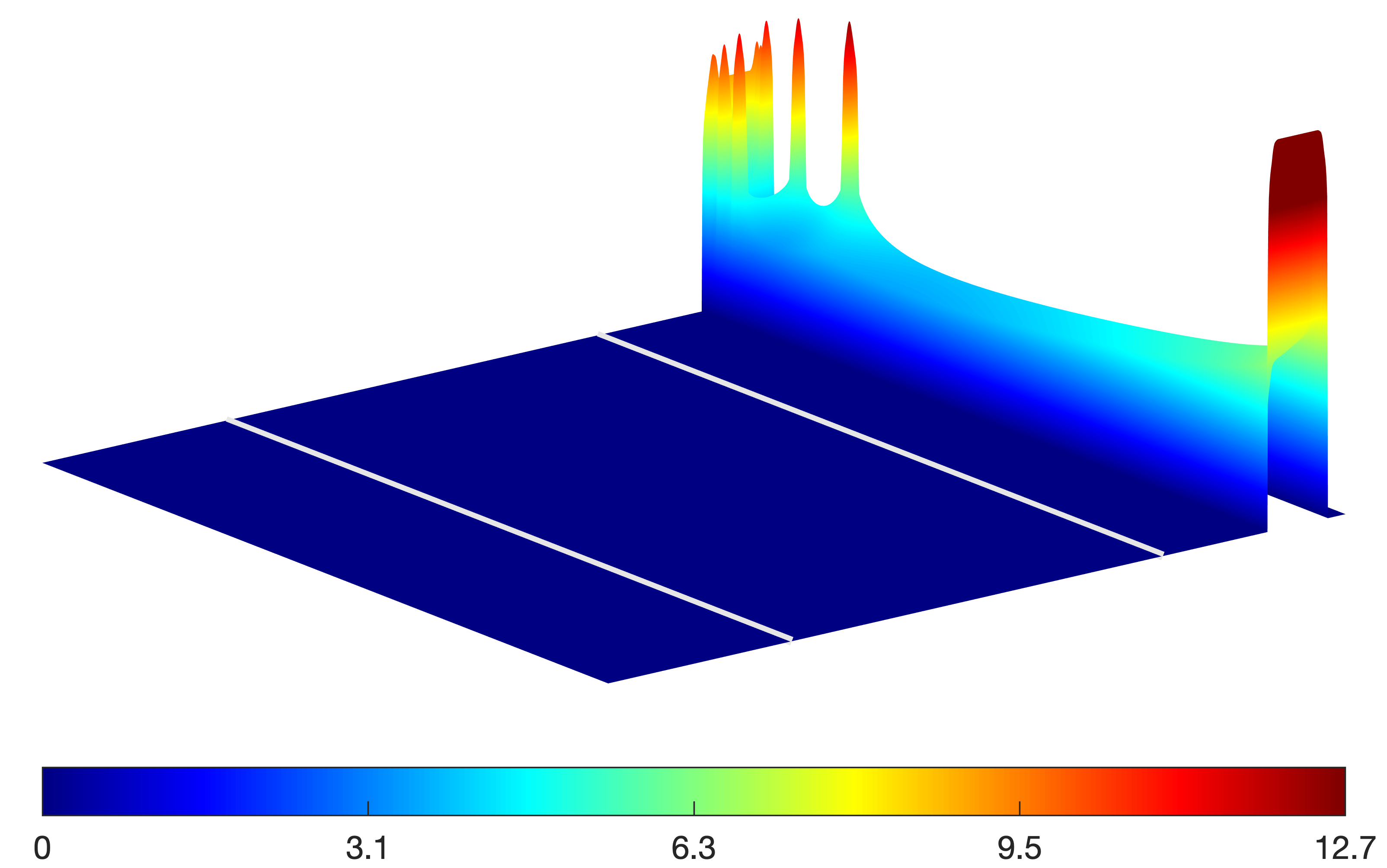}
	\caption{Minimal-time control: intensity of the impulses in logarithmic scale. In the $(x,t)$ plane in blue the time $t$ varies from $t = 0$ (bottom) to $t = T_{\rm min}$ (top).}
	\label{fig3}
\end{SCfigure}

\subsubsection{Lack of controllability when $T < T_{\rm min}$}

In this section, we conclude our discussion on Case 1 by showing the lack of controllability of \eqref{eq-main} when the 
time horizon $T < T_{\rm min}$. 

To this end, we employ a classical gradient method implemented in the DyCon Computational Toolbox (\cite{dycontoolbox}) to solve the following optimization problem: 
\begin{align}\label{opt_dycon}
	\min\;\|u(\cdot,T)-\widehat{u}(\cdot,T)\|^2_{L^2(-1,1)}
\end{align}
subject to the constraints \eqref{min-opt}.

We choose a time horizon $T=0.2 < T_{\rm min}$ and solve the constrained optimization problem \eqref{opt_dycon}.

In Figure~\ref{fig4} we notice that we cannot control the solution to  \eqref{eq-main} any longer. The positive control displayed in Figure~\ref{fig5} is trying to push the initial datum $u_0$ to the desired target but since $T < T_{\rm min}$, we are unable to steer $u_0$ to $\widehat{u}(\cdot,T)$.

\begin{SCfigure}[0.7][h]
	\includegraphics[scale=0.2]{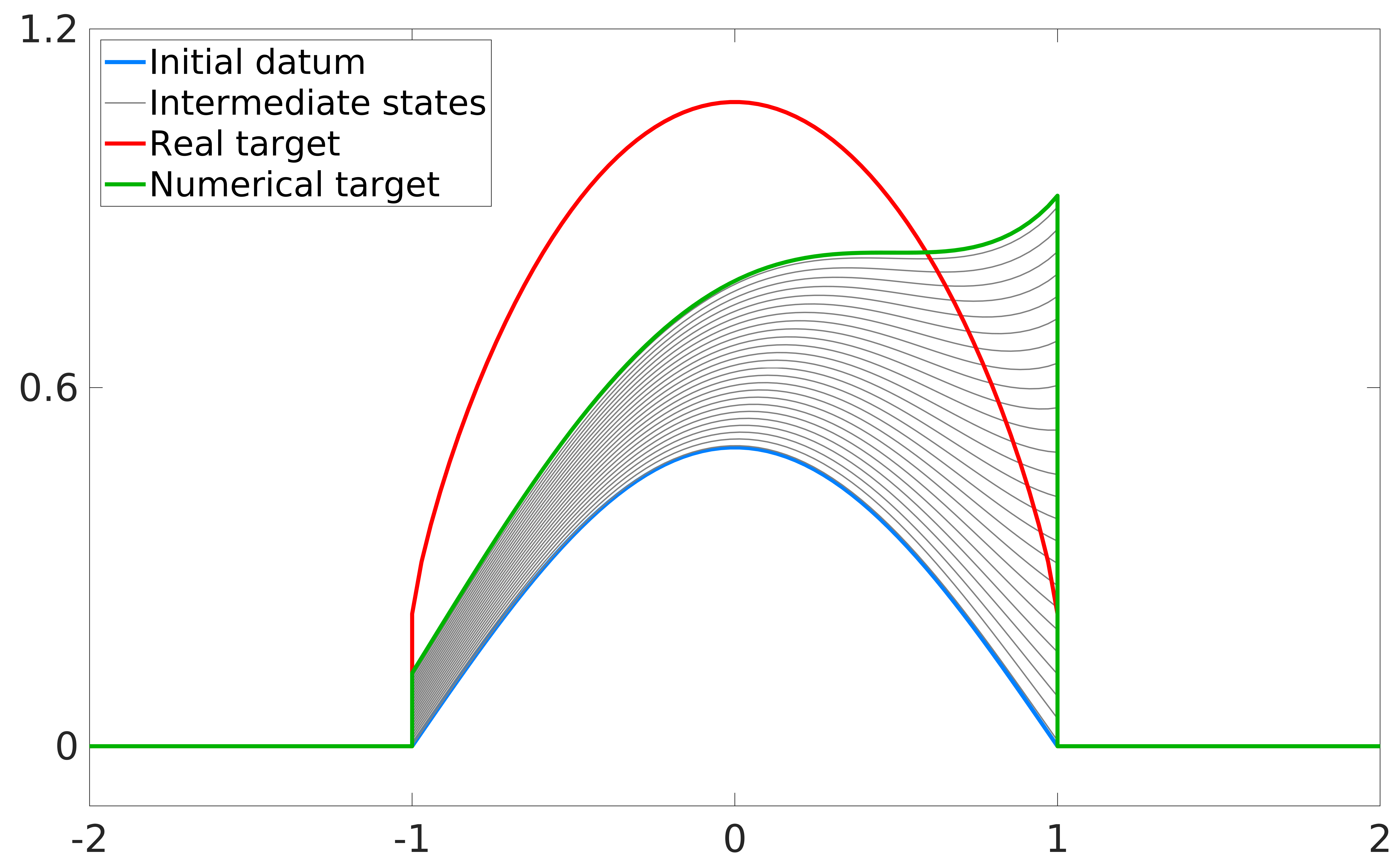}
	\caption{Evolution in the time interval $(0,0.2)$ of the solution to \eqref{main-num} with $s=0.8$ and $n=10^9$. The equation is not controllable to the desired trajectory.}
	\label{fig4}
\end{SCfigure}

\begin{SCfigure}[0.7][h]
	\centering 
	\includegraphics[scale=0.55]{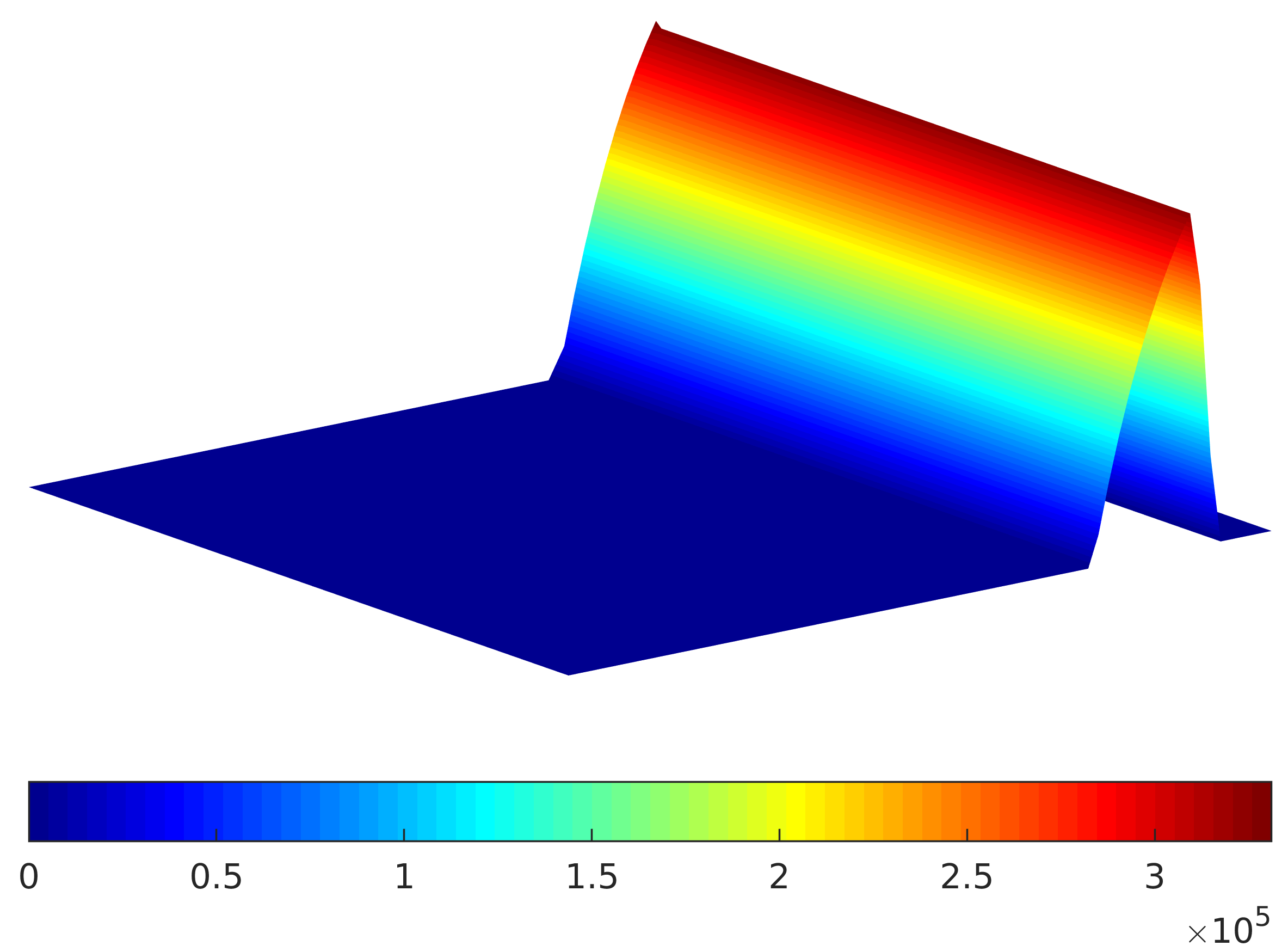}
	\caption{Evolution in the time interval $(0,0.2)$ of the control function computed through the minimization process \eqref{opt_dycon}-\eqref{min-opt}.}
	\label{fig5}
\end{SCfigure}

\subsection{Case 2: $u_0>\widehat{u}(\cdot,T)$}

Let us now consider the case of an initial datum $u_0$ which is greater than the final target $\widehat{u}(\cdot,T)$. As in the previous case, we first solve the optimization problem \eqref{minT}-\eqref{min-opt} using \texttt{CasADi} to determine $T_{\rm min}$. We obtain $T_{\rm min} = 0.5713$. Figure \ref{fig6} shows that in this time horizon the fractional heat equation \eqref{eq-main} is controllable and we can reach $\widehat{u}(\cdot,T)$ from $u_0$. We again observe that the minimal-time control has an impulse nature, see Figures \ref{fig7} and \ref{fig8}.

\begin{SCfigure}[0.4][h!]
	\centering 
	\includegraphics[scale=0.25]{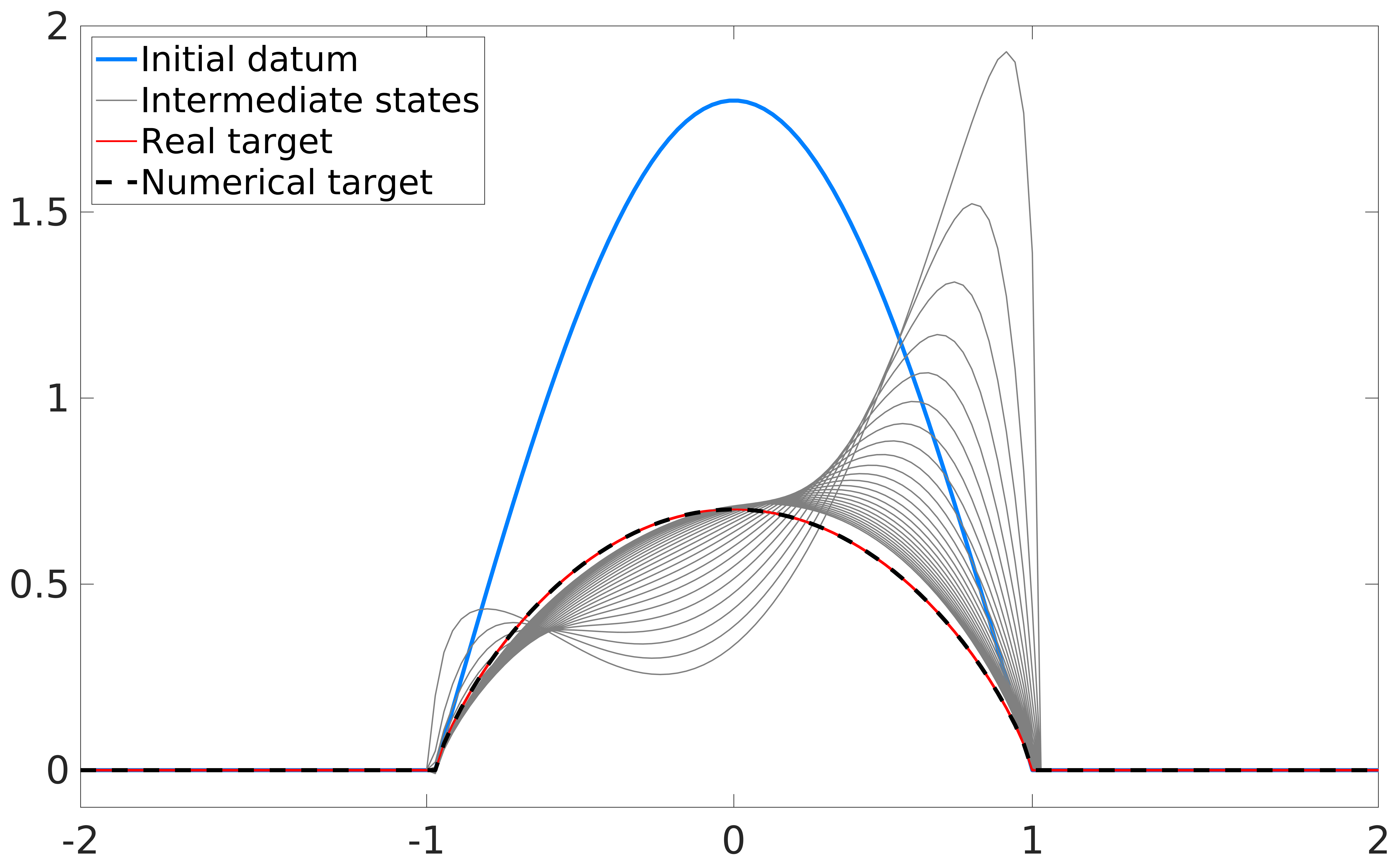}
	\caption{Evolution of the solution to \eqref{HE-approx} in the time interval $(0,T_{\min})$ with $s=0.8$. The blue line is the initial configuration $u_0$. The red line is the target $\widehat{u}(\cdot,T)$ we aim to reach. The black dashed line is the target we computed numerically.}
	\label{fig6}
\end{SCfigure}

\begin{SCfigure}[0.6][h!]
	\includegraphics[scale=0.55]{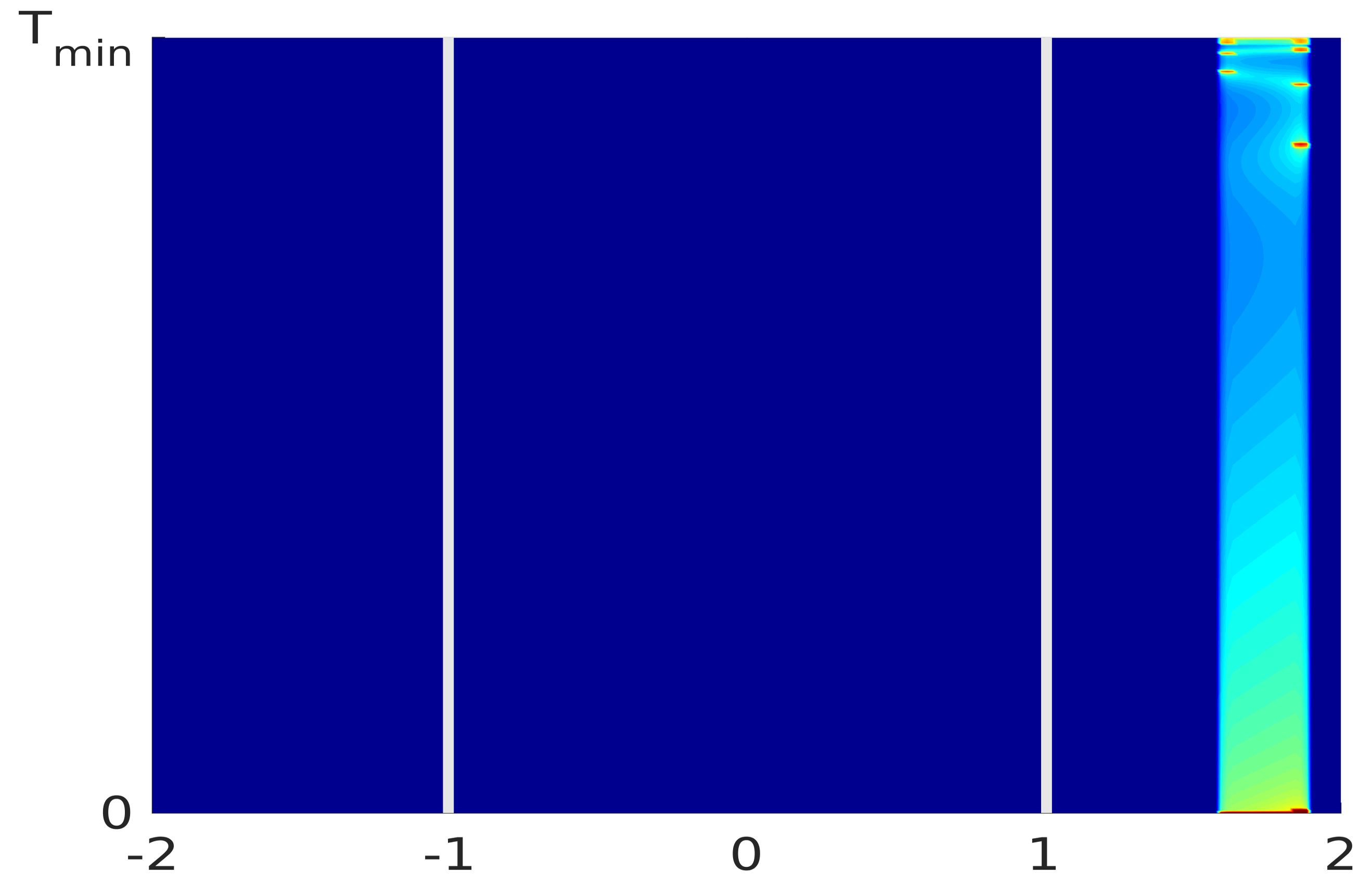}
	\caption{Minimal-time control: space-time distribution of the control. The white lines delimit the dynamics region $(-1,1)$.}
	\label{fig7}
\end{SCfigure}

\begin{SCfigure}[0.6][h!]
	\centering 
	\includegraphics[scale=0.65]{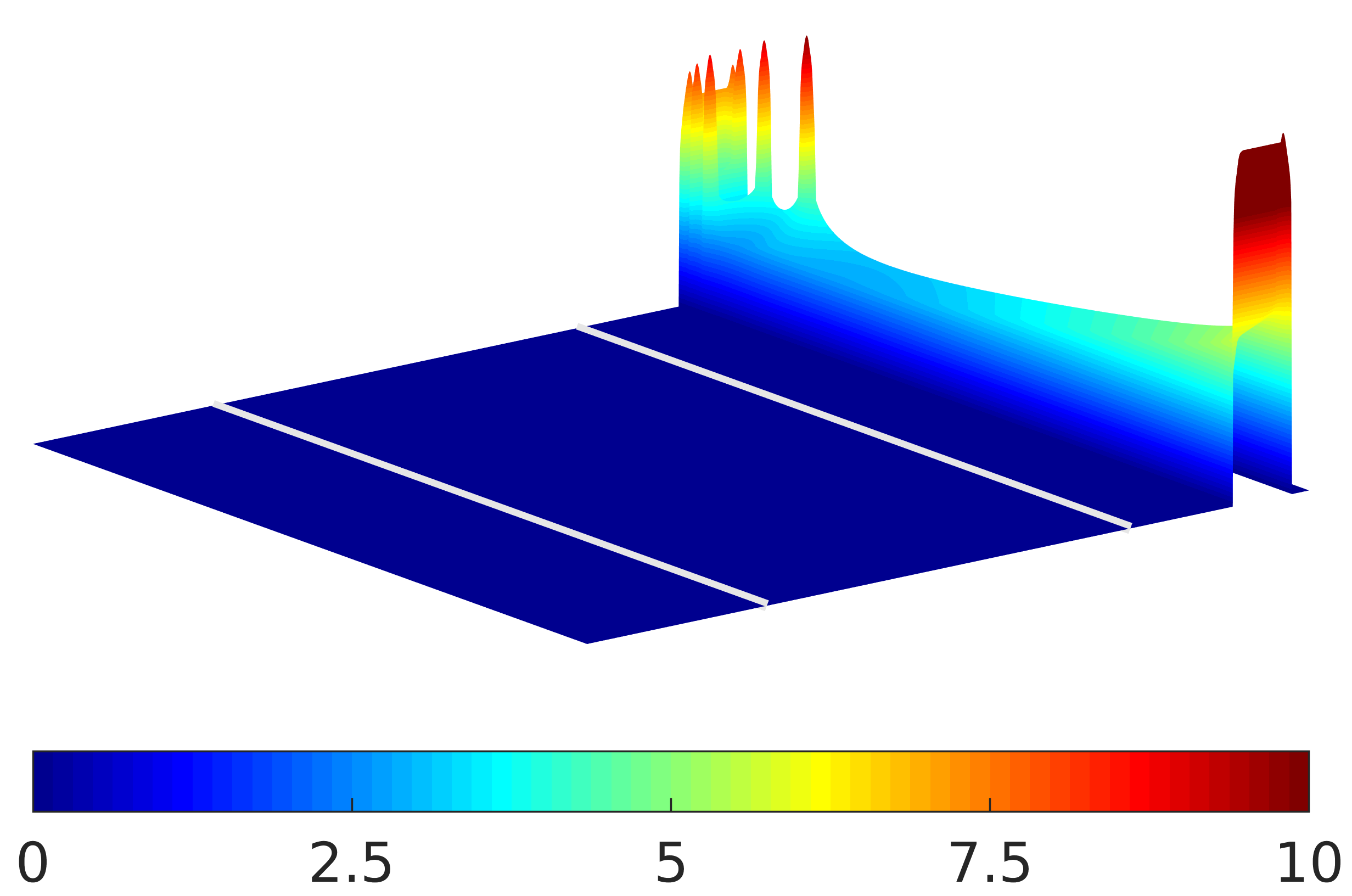}
	\caption{Minimal-time control: intensity of the impulses in logarithmic scale. In the $(x,t)$ plane in blue the time $t$ varies from $t = 0$ (bottom) to $t = T_{\rm min}$ (top).}
	\label{fig8}
\end{SCfigure}

Notice that, this time, we want to reach a target which is below the initial datum $u_0$. To achieve that, the control acts by countervailing the natural dissipation of the fractional heat process, by acting on the solution to \eqref{eq-main} with a positive force. In the end, increases its intensity to reach the desired trajectory.

Since $g$ is not allowed to push itself down (due to the constraints), intuitively we expected to see $g$ to be inactive, at least initially, to 
let the equation dissipate under the action of the heat semigroup. The control becames active only when the solution is close to the final target to do final adjustments. This is what has been observed in \cite{biccari2019controllability} when the control is in the interior of the domain $(-1,1)$. However, our numerical experiments shows that this intuition is no longer valid in the case of the exterior control. This is another example of the fact that the action of the exterior control is very different than the existing notion of interior or boundary controls.

Finally, when considering a time horizon $T<T_{\rm min}$ we again notice that we cannot reach the desired trajectory $\widehat{u}(\cdot,T)$. 
In fact, since we want to reach a final target which is below the initial datum $u_0$, the natural approach is to push down the state with a ``negative'' action. However, since the control is not allowed to do this because of the non-negativity constraint, its best option is to remain inactive for the entire time interval and to let the solution diffuses under the action of the fractional heat semi-group (see Figures \ref{fig9} and \ref{fig10}). But this is not sufficient to reach the target in the time horizon provided.

\begin{SCfigure}[0.6][h!]
	\includegraphics[scale=0.25]{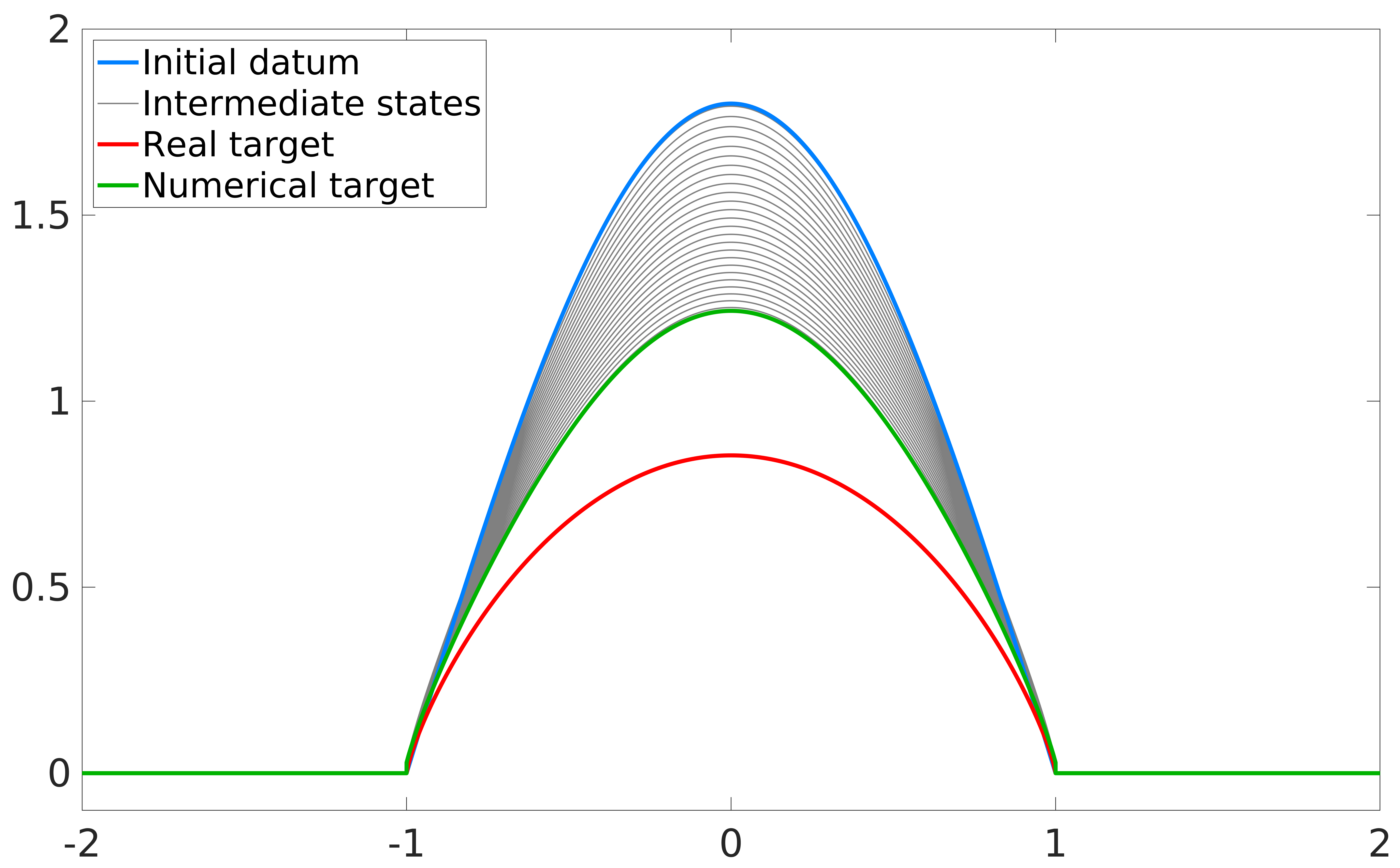}
	\caption{Evolution of the solution to \eqref{HE-approx} in the time interval $(0,T_{\min})$ with $s=0.8$. The blue line is the initial configuration $u_0$. The red line is the target $\widehat{u}(\cdot,T)$ we aim to reach. The green line is the target we computed numerically. The equation is not controllable.}
	\label{fig9}
\end{SCfigure}

\begin{SCfigure}[0.6][h!]
	\centering 
	\includegraphics[scale=0.65]{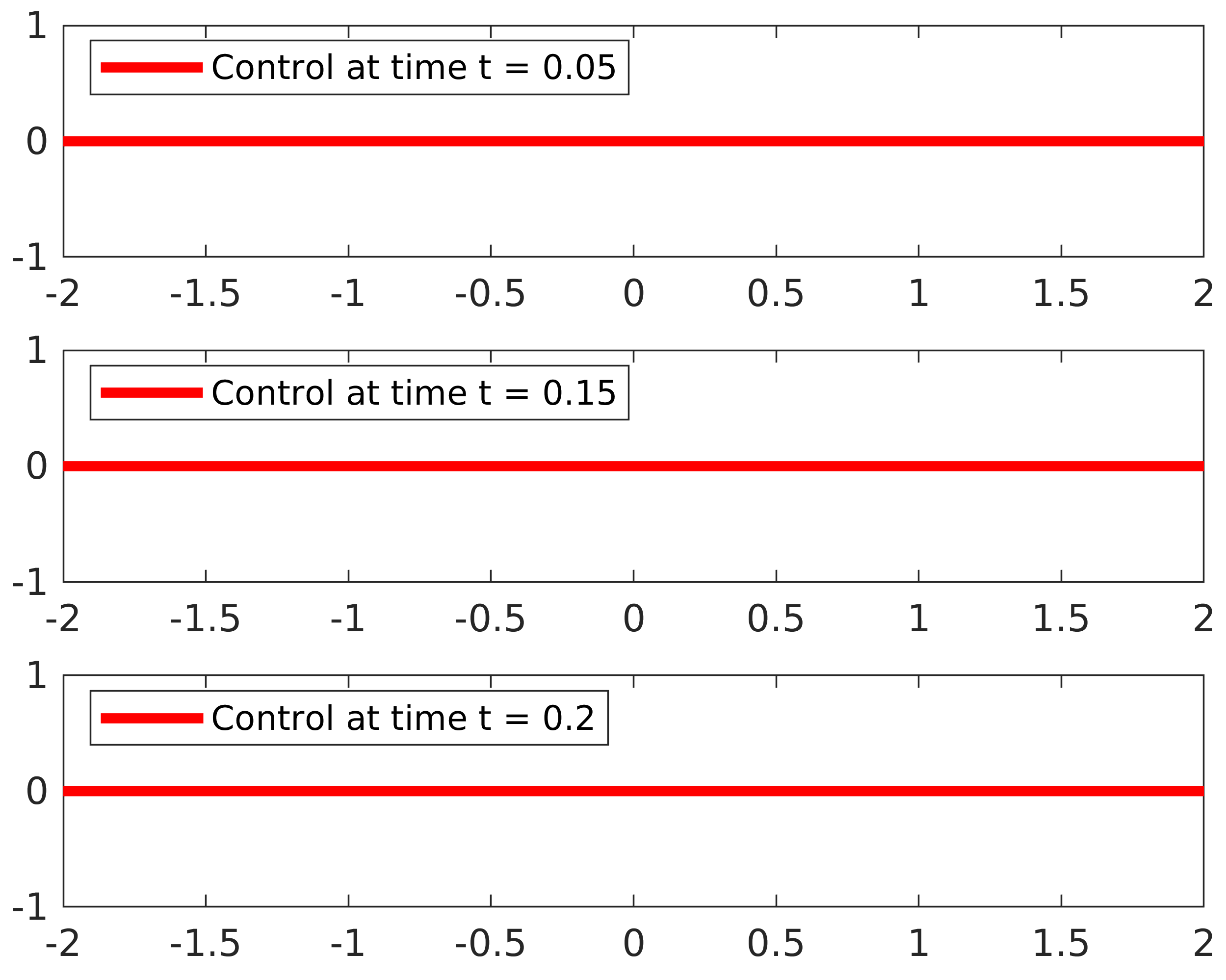}
	\caption{Control corresponding to the dynamics of Figure \ref{fig9}. The control is inactive for the entire time horizon.}
	\label{fig10}
\end{SCfigure}

\section{Concluding remarks}\label{sec-con-rem}

In this paper, we have studied the exterior controllability to trajectories for a one-dimensional fractional heat equation under nonnegativity state and control constraints. This extends our previous analysis presented in \cite{biccari2019controllability} for the case of interior controls.

For $s>1/2$, when the interior and exterior controllabilities for the unconstrained fractional heat equation holds in any positive time $T>0$, we have shown that the introduction of state or control constraints creates a positive minimal time $T_{\rm min}$ for achieving the same result. Moreover, we have also proved that, in this minimal time, exterior constrained controllability holds with controls in the space of Radon measures. 

Our results, which are in the same spirit of the analogous ones obtained in \cite{biccari2019controllability,loheac2017minimal,pighin2018controllability}, are supported by the numerical simulations in Section \ref{sec-num}.

We present hereafter a non-exhaustive list of open problems and perspectives related to our work.

\begin{itemize}
	\item[1.] \textbf{Extension to the multi-dimensional case.} Our analysis, based on spectral techniques, applies only to a one-dimensional fractional heat equation. The extension to multi-dimensional problems on bounded domains $\Omega\subset\RR^N$, $N\geq 1$ is still completely open, even in the unconstrained case. This would require different tools such as Carleman estimates. Nevertheless, obtaining Carleman estimates for the fractional Laplacian is a very difficult issue which has been considered only partially, and only for problems defined on the whole Euclidean space $\RR^N$ (see, e.g., \cite{ruland2015unique}). The case of bounded domains remains currently unaddressed and it is quite challenging. As one expects, the main difficulties come from the nonlocal nature of the fractional Laplacian, which makes classical PDEs techniques more delicate or even impossible to use. 
	
%
	\item[2.] \textbf{Lower bounds for the minimal constrained controllability time.} In Section \ref{sec-num}, we gave some numerical lower bound for the minimal constrained controllability time. Nevertheless, we cannot ensure that the bounds we presented are optimal. This raises the very important issue of obtaining analytical lower bounds for the controllability time. In particular, to understand how it depends on the order $s$ of the fractional Laplacian is evidently a fundamental point to be clarified. This question was already addressed in \cite{loheac2017minimal,pighin2018controllability} for the local heat equation but, as we discussed in \cite[Section 4,4]{biccari2019controllability}, the methodology developed in those works does not apply immediately to our case. Therefore, there is the necessity to adapt the techniques of \cite{loheac2017minimal,pighin2018controllability}, or to develop new ones. 
	
	\item[3.] \textbf{Convergence result for the minimal time.} The minimal time $T_{\rm min}$ in the simulations of Section \ref{sec-num} is just an approximation computed by solving numerically the optimization problem \eqref{minT}-\eqref{min-opt}. The validity of these computational result should be confirmed by showing that this minimal time of control for the discrete problem converges towards the continuous one as the mesh-sizes tend to zero. This could be done by adapting the procedure presented in \cite[Section 5.3]{loheac2017minimal}. Nevertheless, we have to mention that, in order to corroborate this procedure, it is required the knowledge of an analytic lower bound for $T_{\rm min}$ which, at the present stage, it is unknown (see point 2 above).\\
\end{itemize}

\noindent
{\bf Acknowledgement}:
 Part of this research was carried out while the fifth author (SZ) visited DeustoTech and the University of Deusto, Bilbao, Spain, with the financial support of the DyCon project. He would like to thank the members of this institution for their kindness and warm hospitality.

\bibliographystyle{plain}
\bibliography{biblio}

\end{document}